\documentclass[11pt]{article}

\usepackage{amsmath,amsfonts,amssymb,amsthm,bbm,latexsym,mathrsfs}
\usepackage{graphicx,color,epsfig,fancyhdr,dsfont, ulem}
\usepackage{enumerate}
\usepackage{hyperref}
\usepackage{indentfirst}
\usepackage{amsmath,amscd}
\usepackage{caption}
\usepackage{graphicx, subfig}
\usepackage{tikz-cd}
\usepackage[all]{xy}
\usepackage[dvipsnames]{xcolor}
\usepackage{amsmath, amsfonts, amssymb, mathrsfs, enumerate, enumitem, titlesec}
\usepackage[titletoc]{appendix}
\usepackage{amsthm}

\title{Newelski’s Conjecture for $o$-Minimal and $p$-Adic Groups}
\date{}

\date{\today}

\author{Ningyuan Yao and Zhentao Zhang}

\setenumerate[1]{itemsep=0pt,partopsep=0pt,parsep=\parskip,topsep=5pt}

\def\Gm{\mathbb{G}_{\mathrm{m}}}
\def\nn{{\mathcal{N}}}
\def\Zdim{\textnormal{Zdim}}

\def\tp#1{\textnormal{tp}(#1)}
\newcommand\tpo[3][]{\textnormal{tp}^{#1}\left(#2\middle/#3\right)}

\newcommand\dcl[2][]{\textnormal{dcl}^{#1}(#2)}

\newcommand\cl[2][]{\textnormal{cl}^{#1}(#2)}
\newcommand{\clz}[1]{\operatorname{cl}_z (#1)}

\def\alg{\textnormal{alg}}

\newcommand\Q{{\mathbb Q}_p}

\def\r{\mathbb{R}}

\def\pM{p^*_{M}}
\def\pM0{p^*_{M_0}}

\def\pcf{p\mathrm{CF}}

\def\Gen{\mathrm{Gen}}
\def\ext{\mathrm{ext}}
\def\j{\mathcal{J}}
\def\M{\mathbb{M}}

\def\N{\mathbb{N}}

\def\pcf{p\textnormal{CF}}
\def\Th{\textnormal{Th}}
\def\Def{\textnormal{Def}}

\def\GL{\textnormal{GL}}
\def\SL{\textnormal{SL}}
\def\PSL{\textnormal{PSL}}

\def\ext{\textnormal{\lowercase{ext}}}

\def\id{\text{id}}
\def\ic{{\cal M}_{_C}^{^M}}
\def\ig{{\cal M}_{_G}^{^M}}
\def\ec{{\cal I}_{_C}^{^M}}
\def\ew{{\cal I}_{_W}^{^M}}
\def\eb{{\cal I}_{_B}^{^M}}
\def\ex{{\cal I}_{_X}^{^M}}
\def\eg{{\cal I}_{_G}^{^M}}
\def\ee{{\cal I}_{_E}^{^M}}
\def\ey{{\cal I}_{_Y}^{^M}}

\def\ih{{\cal M}_{_H}^{^M}}
\newcommand\stb[2][]{\textnormal{Stab}_{#1}(#2)}

\def\st{\textnormal{st}}
\def\sq{\subseteq}
\def\Ind#1#2{#1\setbox0=\hbox{$#1x$}\kern\wd0\hbox to 0pt{\hss$#1\mid$\hss}
	\lower.9\ht0\hbox to 0pt{\hss$#1\smile$\hss}\kern\wd0}

\def\notind#1#2{#1\setbox0=\hbox{$#1x$}\kern\wd0
	\hbox to 0pt{\mathchardef\nn=12854\hss$#1\nn$\kern1.4\wd0\hss}
	\hbox to 0pt{\hss$#1\mid$\hss}\lower.9\ht0 \hbox to 0pt{\hss$#1\smile$\hss}\kern\wd0}

\newtheorem{dfn}{Definition}[subsection]

\newtheorem*{Claim}{Claim}

\newtheorem{conj}{Conjecture}

\newtheorem*{KTC}{Kneser-Tits Conjecture}
\newtheorem*{NC}{Newelski's Conjecture}
\newtheorem{theorem}{Theorem}
\newtheorem{fact}[dfn]{Fact}
\newtheorem{lemma}[dfn]{Lemma}
\newtheorem{claim}[dfn]{Claim}
\newtheorem{pro}[dfn]{Proposition}
\newtheorem{cor}[dfn]{Corollary}

\newtheorem{rmk}[dfn]{Remark}
\newtheorem{notations}[dfn]{Notations}
\newtheorem*{theorem-star}{Main Theorem}
\newtheorem*{theorem-star-A}{Theorem A}
\newtheorem*{theorem-star-B}{Theorem B}
\newtheorem*{theorem-star-C}{Theorem C}
\newtheorem*{theorem-star-D}{Theorem D}

\begin{document}
	\bibliographystyle{siam}
	
	\maketitle
	


\begin{abstract}
Let $M_0$ denote either the field of $p$-adic numbers $\mathbb{Q}_p$ or an $o$-minimal expansion of the real field $\mathbb{R}$. We study the minimal flows and Ellis groups of groups definable over $M_0$ from the viewpoint of definable topological dynamics. The computation of Ellis groups, together with the closely related Newelski conjecture, has so far been available only under restrictive hypotheses: in the $p$-adic setting, the results of \cite{BY-APAL} apply to reductive algebraic groups, while in the $o$-minimal setting the known computations rely on a compact-torsion-free (Iwasawa) decomposition. In this paper we remove both restrictions.

We show that every definable group $G$ has a definably amenable radical $D$, with quotient map $\pi:G\to A$ onto a centreless semisimple group $A$. In $A$, a definably amenable component $V$ is selected: a maximal definably amenable subgroup containing a maximal definably solvable subgroup. Its preimage $B=\pi^{-1}(V)$ is a definably amenable component of $G$, contains $D$, and is the carrier for the computation of the Ellis group.

Our main result is that, for every $M\succ M_0$, the Ellis group of the universal definable flow of $G$ over $M$ is isomorphic to that of $B$ over $M$, and hence to $B/B^{00}$. In particular, the Ellis groups of $G$ are model-independent. As a consequence, when $G$ is definable over a $p$-adically closed field, Newelski's conjecture holds for $G$ if and only if $G$ is definably amenable. The amenable direction is the known result of \cite{CS-Definably-Amenable-NIP-Groups}, while the converse is new.
\end{abstract}

\section{Introduction}
	
We investigate the definable topological-dynamical properties of groups definable in $p$-adically closed fields and in $o$-minimal expansions of real closed fields. Our framework is the model-theoretic topological dynamics initiated by Newelski \cite{Newelski-TD-gp-Act}. Let $G$ be a group definable in a monster model $\M$ over a small submodel $M\prec \M$. The natural action of $ G(M) $ on $S_G(M) $—the space of complete types over $M$ concentrating on $G$—endows $S_G(M) $ with the structure of a $ G(M) $-flow. The ultrafilter space $S_{\text{ext},G}(M)$ of externally definable subsets of $ G(M) $ carries a canonical semigroup operation $*$, and $(S_{\text{ext},G}(M),*)$ is canonically isomorphic to the enveloping semigroup of the $ G(M) $-flow $S_G(M) $; we accordingly call $S_{\text{ext},G}(M)$ the \emph{universal definable flow} of $G$ over $M$. It contains a family of pairwise isomorphic subgroups, henceforth \emph{ideal groups}, whose common isomorphism type is the \emph{Ellis group} of $S_{\text{ext},G}(M)$ (see subsection \ref{Definable topological dynamics} for details).

Newelski revealed the following relationship between the Ellis group of $S_{\text{ext},G}(M)$ and the $M$-type definable connected component $G_M^{00}$ of $G$:   Let ${\cal I} \subseteq S_{\text{ext},G}(M)$ be an ideal group. The map $f: {\cal I} \to G/G_M^{00}$ defined by $p \mapsto p/G_M^{00}$ is an epimorphism, where $p/G_M^{00}$ denotes the unique coset of $G_M^{00}$ determined by $p$,  i.e., the coset $C$ satisfying $p\vdash (x\in C)$ (see \cite{Newelski-TD-gp-Act},  Proposition 4.4).

When $ G $ is a stable group, $ S_{\text{ext},G}(M) $ coincides with $ S_G(M) $, $ G_M^{00} $ is precisely $ G^{0} $, the  (unique) ideal group is exactly the space of generic types in $ S_G(M) $, and the above homomorphism is a group homeomorphism.  
In this context, the Ellis group can be viewed as a natural generalization of the generic type space concept, allowing us to translate fundamental theorems of stable groups to certain tame non-stable groups. In  \cite{Newelski-TD-gp-Act}, Newelski conjectured that in tame settings, such as NIP theories, the Ellis group $E$  of $S_{G,\ext}(M)$ is abstractly isomorphic to $ G/G_M^{00} $. Note that $G_M^{00}$ is model-independent in NIP theories, so we denote it by $G^{00}$.   We formulate Newelski's conjecture as follows:
\begin{NC}
Let $G$ be an NIP group and  $ {\cal I} \subseteq S_{\text{ext},G}(M) $ be an ideal group. Then the map $ p \mapsto p/G^{00} $ is an isomorphism from $ {\cal I}$ to $ G/G^{00} $.
\end{NC}

In particular, Newelski's conjecture implies that ideal groups are independent of the choice of model $ M $. A great deal of research has been motivated by this conjecture.

For convenience, throughout this paper, an ``\emph{$o$-minimal theory}'' means a theory of an $o$-minimal expansion of the ordered field $\mathbb{R}$; {\em an ideal group of $G$ over $M$} means an ideal group of $S_{\text{ext},G}(M)$; and the {\em Ellis group of $G$ over $M$} has the same meaning. We say that {\em the Ellis group of $G$ is model-independent} if, for any two models $M,N$ over which $G$ is defined, the Ellis group of $G$ over $M$ is isomorphic to that over $N$.

In the $o$-minimal setting, the first counterexample to Newelski's conjecture was given in \cite{slr}, where the conjecture is shown to fail for $\SL(2,\mathbb{R})$. Recall that an $M$-definable group $G$ admits a compact--torsion-free decomposition over $M$ if $G = KH$ for $M$-definable subgroups $K,H$ such that $K$ is definably compact, $H$ is torsion-free, and $K \cap H = \{\id_G\}$. The group $\SL(2,\r)$ admits such a decomposition with $K = \mathrm{SO}(2,\mathbb{R})$ and $H$ the Borel subgroup of upper triangular matrices, and \cite{slr} proves that the Ellis group of $S_{\SL(2,\r)}(\r)$ is isomorphic to $(N_G(H) \cap K)(\r)$, a cyclic group of order two. Jagiella generalized this result to every group definable in an $o$-minimal expansion of $\r$ that admits such a decomposition: its Ellis group over $\r$ is isomorphic to $(N_G(H) \cap K)(\r)$, so $\SL(n,\r)$ and $\GL(n,\r)$ are counterexamples for all $n \geq 2$ \cite{Jagiella-I}. Subsequently, the first author of this paper showed that the same conclusion holds over any model: if $G$ is definable in an $o$-minimal expansion of $\mathbb{R}$ with such a decomposition over $\r$, then for every $M \succ \r$, the Ellis group of $G$ over $M$ is again isomorphic to $(N_G(H) \cap K)(\r)$ \cite{YL-APAL}.
This result demonstrates that the Ellis group remains model-independent even in scenarios where Newelski's conjecture fails. Moreover, the paper \cite{Jagiella-Invs} generalizes this property to a broader framework: Let $ M $ be an $o$-minimal expansion of a real closed field, and let $ G $ be a group definable over $ M $. For any $ N \succ M $, the Ellis group of $ G $ over $ M $ is isomorphic to that over $ N $. Furthermore, by appealing to Pillay's Conjecture, both groups are abstractly isomorphic to some subgroup of a compact Lie group. Notably, none of the aforementioned counterexamples are definably amenable.

In   \cite{CPS-Ext-Def-in-NIP}, the authors strengthened the ``tame" condition in Newelski's conjecture to NIP-definable amenability. 
 In that very paper, several special cases of this conjecture were established, including  definably extremely amenable, fsg groups, dfg groups,  definably amenable dp-minimal groups, and definably amenable groups definable in an $o$-minimal expansion of a real closed field. Subsequently, Chernikov and Simon proved the full conjecture in \cite{CS-Definably-Amenable-NIP-Groups}:

\begin{fact}[\cite{CS-Definably-Amenable-NIP-Groups} Theorem 5.6]\label{fact-NIP-Amenable-CS}
Let $ G $ be a definably amenable group definable over an NIP structure $ M $. Then Newelski's conjecture holds for $ G $.
\end{fact}


In the $p$-adic setting, $\SL(2,\Q)$ has been investigated in \cite{PPY-sl2qp}, which provides yet another counterexample to Newelski's conjecture. Subsequently, a follow-up study in \cite{BY-APAL} further explores reductive algebraic groups over $\Q$, thereby establishing that all non-definably compact simple algebraic groups over $\Q$ constitute counterexamples to Newelski's conjecture. Notably, however, the Ellis groups remain model-independent in this context.  
 
 In \cite{Kirk}, Kirk established that $\SL(2,\mathbb{C}((t)))$ also serves as a counterexample to Newelski's conjecture. In fact, the principal results from the literatures \cite{YL-APAL,PPY-sl2qp,Kirk} collectively demonstrate that for $M$ being $\mathbb{R}$, $\mathbb{Q}_p$, or $\mathbb{C}((t))$, and $G = \SL(2,M)$, the Ellis group of $G$ over any $N \succ M$ is isomorphic to the Ellis group of its Borel subgroup $B$, specifically $B/B^{00}$. The authors of \cite{BY-APAL} proved that for a reductive algebraic group $G$ over $\mathbb{Q}_p$, the Ellis group of $G$ over any model $N \succ \Q$ is likewise isomorphic to the Ellis group of its Borel subgroup $B$.  Jagiella has shown that if a field $k$ satisfies the NIP property, there exists an epimorphism from the Ellis group of $\SL(2,k)$ to $B/B^{00}$. 

This paper presents a method for computing the Ellis group of any definable group in both $o$-minimal and $ p $-adic settings. This approach facilitates the computation of definable groups in these two settings within a unified framework, generalizing all the aforementioned results.

We now highlight the difference between our approach and that of \cite{Jagiella-Invs}. There, definably extremely amenable groups play a crucial role: in NIP theories such groups satisfy $G = G^{00}$, and in $o$-minimal theories they are precisely the torsion-free groups. The authors view definably extremely amenable groups as the NIP-theoretic analogue of torsion-free groups in the $o$-minimal setting, and regard the ``fsg--definably-extremely-amenable'' decomposition as an abstract definable counterpart of the compact--torsion-free decomposition (more generally, the Iwasawa decomposition), thereby extending the method for computing Ellis groups from $o$-minimal theories to NIP theories.

In contrast, this paper develops another direction in which torsion-free groups of $o$-minimal theories can be generalized, namely dfg groups (see Remark \ref{rmk/def-fsg-dfg}). In $o$-minimal theories, torsion-free groups, connected dfg groups, and triangulable algebraic groups are (eventually) equivalent (see Fact \ref{fact-dfg-group-R}); in $p$-adically closed fields, dfg groups are likewise (eventually) equivalent to triangulable algebraic groups (see Fact \ref{fact-dfg-group-Qp}). Thus dfg groups characterize triangulable algebraic groups simultaneously in both settings. By contrast, in $p$-adically closed fields definably extremely amenable groups can only be unipotent algebraic groups, whereas the multiplicative group $\mathbb{G}_m$ and its finite-index subgroups are not definably extremely amenable.

Our results show that, even in $o$-minimal theories, the Iwasawa decomposition is not necessary: to compute the Ellis groups, what is essentially required is an understanding of the dfg component, which is, up to isogeny (equivalently, commensurability), a maximal dfg subgroup of $G$ (see Definition \ref{def-dfg component}). This approach also yields an explicit construction of an ideal group.

Let $M_0$ be either the field $\mathbb{Q}_p$ of $p$-adic numbers or an $o$-minimal expansion of the field $\mathbb{R}$ of real numbers, let $M \succ M_0$ be a small elementary extension, let $\mathbb{M} \succ M_0$ be a monster model, and let $G$ be a group definable in $\mathbb{M}$. With this notation, our main results are as follows.

\begin{theorem-star-A}[= Theorem \ref{theorem-rG*eb}]
Suppose that $G$ is definable over $M_0$, and let $B \leq G$ be a definably amenable component of $G$ defined over $M_0$. Then the Ellis group of $G$ over $M$ is isomorphic to the Ellis group of $B$ over $M$, which in turn is isomorphic to $B/B^{00}$. In particular, the Ellis group of $G$ is model-independent.
\end{theorem-star-A}

Applying Theorem A, we obtain the converse of Fact \ref{fact-NIP-Amenable-CS} in the $p$-adic setting.
\begin{theorem-star-B}[= Theorem \ref{thm-Newelski's Conjecture=DA}]
Let $M$ be a $p$-adically closed field and $G$ a group definable over $M$. Then $G$ is definably amenable if and only if Newelski's conjecture holds for $G$, i.e. an ideal group of $G$ over $M$ is isomorphic to $G/G^{00}$ via the map $p \mapsto p/G^{00}$.
\end{theorem-star-B}


The paper is organized as follows. Section 2 recalls the notation, definitions, and results from earlier work needed below, including definable topological dynamics and definably amenable, dfg, and fsg NIP groups. In the $o$-minimal and $p$CF ($p$-adically closed field) settings, we also cover definable manifolds, definable compactness, and decompositions of definable groups in terms of their dfg and fsg components. The section closes with a discussion of algebraic groups, where we prove that if $M$ is a $p$-adically closed field or a real closed field, then an $M$-defined algebraic group $G$ is $M$-simple if and only if it is $N$-simple for every $N \succ M$.

Section 3 contains the main results, concerning groups definable in $\mathbb{M}$ over $M$, where $\mathbb{M}$ is a monster model of $\operatorname{Th}(M)$ and $M$ is either a $p$-adically closed field or an $o$-minimal expansion of a real closed field. It is divided into seven subsections.

The first subsection presents the definably-amenable--semisimple decomposition: every definable group $G$ fits into a short exact sequence $1 \to D \to G \overset{\pi}{\to} S \to 1$, where $D$ is the definably amenable radical of $G$ (a maximal definably amenable normal subgroup; see Proposition \ref{pro-definably amenable radical}), $S(M)$ is an open subgroup of $\mathbf{S}(M)$, and $\mathbf{S}$ is a centreless semisimple algebraic group over $M$. If, in addition, $G$ is not definably amenable, then $\mathbf{S}$ is $M$-isotropic and $S(M)$ has finite index in $\mathbf{S}(M)$. This motivates the following conjecture.

\begin{conj}
Let $G$ be an NIP group. Then $G$ admits a maximal definably amenable normal subgroup, unique up to commensurability.
\end{conj}

The second subsection studies the dfg component and definably amenable component of $S$, and proves that the preimage $B$ of a definably amenable component of $S$ under $\pi \colon G \to S$ is precisely a definably amenable component of $G$.

The third subsection investigates the action of $G$ on the definably compact space $G/H$, where $H$ is a dfg component of $G$ contained in $B$. We show that if $p \in S_{G,\ext}(M)$ is $H^0(M)$-invariant, then $p \vdash \mu_G B$, where $\mu_G$ is the type-definable group of infinitesimals of $G$ over $M$.

The fourth subsection constructs an ideal group of a definably amenable group $B$: if $H$ is a normal subgroup of $B$ and $C = B/H$, an ideal group of $B$ can be built from those of $H$ and $C$.

The fifth subsection treats the standard model and proves that $\mu_G$ always admits an fsg type, denoted $r_G$. This allows $\mu_G$ to replace the definably open fsg subgroup used in \cite{BY-APAL}: such an open subgroup exists in the $p$-adic setting but not in the $o$-minimal setting, where $\mu_G$ serves as the substitute.

In the sixth subsection, we construct an ideal group of $G$ over $M$: let $\eb$ be an ideal group of $B$ over $M$; then $r_G * \eb$ is an ideal group of $G$ over $M$, and the two are shown to be isomorphic.

The final subsection establishes a necessary and sufficient condition for Newelski's conjecture, namely $B^{00} = G^{00} \cap B$. We also show that if $f \colon G \to X$ is a definable surjective homomorphism of definable groups, then the conjecture for $G$ implies it for $X$. These results are then used to prove that a definable group over a $p$-adically closed field satisfies Newelski's conjecture if and only if it is definably amenable.

\section{Preliminaries}

\subsection{Notations}

We assume basic familiarity with model theory and refer to \cite{Pzt-book} and \cite{Maker-book} as standard references. Let $T$ be a complete theory with infinite models in a language $L$, and let $\M$ be its monster model, a sufficiently saturated model in which every type over a small subset $A \subseteq \M$ (where ``small'' means $|A| < |\M|$) is realized. Small elementary submodels of $\M$ are denoted by $M,N$. The variables $x,y,z$ range over finite tuples, and $a,b,c$ range over tuples in $\M^n$ for $n \in \mathbb{N}$.

Unless otherwise specified, all definable sets are taken in $\M$. For an $L_\M$-formula $\phi(x)$ and $B \sq \M$, write
\[
\phi(B) = \{b \in B^{|x|} \mid \M \models \phi(b)\}.
\]
We do not distinguish notationally between a definable set and its defining formula, and always regard $D$ as defined by the formula $D(x)$. For $B \subseteq \M$, we therefore write $D(B)$ for the set of realizations of $D(x)$ in $B$. For a partial type $\eta(x)$, the expressions $D \in \eta$, $\eta \vdash D$, and $\eta \subseteq D$ abbreviate respectively $D(x) \in \eta$, $\eta \vdash D(x)$, and $\eta(\bar{\M}) \subseteq D(\bar{\M})$, where $\bar{\M}$ is an $|\M|^+$-saturated elementary extension of $\M$. Note that $\eta \vdash D$ and $\eta \subseteq D$ are equivalent.

For an $M$-definable set $X$, $S_X(M)$ denotes the space of complete types over $M$ concentrating on $X$, i.e. $\{p \in S(M) \mid p \vdash X\}$. Definable types, heirs, and coheirs are used in the standard sense (see \cite{Pzt-book}). For $A,B \subseteq \M$ and $p \in S(A)$, $p \upharpoonright B$ denotes the restriction of $p$ to $B$ when $A \supseteq B$, whereas $p|B$ denotes its unique heir over $B$ when $B \supseteq A$ (with $A$ a model and $p$ definable).


\subsection{Definable topological dynamics}\label{Definable topological dynamics}

\subsubsection*{$G$-flows and its Ellis semigroup}
Let us first briefly recall the basics of classical topological dynamics. Let $G$ be a (discrete) group. A (point-transitive) {\em $G$-flow} is a pair $(X, G)$ consisting of a compact Hausdorff space $X$ and a left action of $G$ on $X$ by homeomorphisms, with the additional property that $X$ admits a dense $G$-orbit. A subset $Y \subseteq X$ is called a {\em subflow} if $Y$ is a closed, $G$-invariant subspace of $X$ (i.e., $g \cdot Y \subseteq Y$ for all $g \in G$). A subflow $Y$ is {\em minimal} if it has no proper subflows; this is equivalent to the condition that $Y = \cl{G \cdot y}$ for every $y \in Y$. A point $x \in X$ is said to be {\em almost periodic} if $x$ lies in some minimal subflow of $X$.

For each $g \in G$, let $\pi_g: X \to X$ denote the homeomorphism induced by the group action, i.e., $\pi_g(x) = g \cdot x$. Let $X^X$ be the set of all functions from $X$ to $X$, endowed with the product topology (compact Hausdorff by Tychonoff's theorem). Define $E(X)$ as the closure of $\{\pi_g : g \in G\}$ in $X^X$. Equipped with function composition as the operation $*$, $E(X)$ forms a semigroup. The $G$-action on $E(X)$ given by $g \cdot f = \pi_g * f$ turns $E(X)$ into a $G$-flow, which is naturally isomorphic to its own Ellis semigroup. For any $x \in X$, the closure of its $G$-orbit satisfies $\cl{G \cdot x} = E(X)(x) = \{f(x) : f \in E(X)\}$. 

\subsubsection*{Minimal subflows and ideal groups}
In $E(X)$,  minimal subflows coincide with minimal left ideals, and any two such minimal subflows are homeomorphic as $G$-flows. We sometimes use the phrase ``minimal subflow of $E(X)$" to denote the homeomorphism class of these minimal subflows. A minimal subflow $I$ is the closure of the $G$-orbit of every $p \in I$. 
The following fact is easy to verify:
\begin{fact}\label{fact-minimal-ideal}
    $I\sq E(X)$ is a minimal subflow iff for any $q\in E(X)$ and $p\in I$, $p\in E(X)*q*p$.
\end{fact} 
An element $u \in I$ is called an {\em idempotent} if $u * u = u$, and we denote the set of all idempotents in $I$ by $\mathrm{Idem}(I)$.  For any $u \in \mathrm{Idem}(I)$, the pair $(u * I, *)$ forms a group with $u$ as its identity element. Moreover, $I$ can be decomposed as a disjoint union of the groups $u * I$ for $u \in \mathrm{Idem}(I)$. Notably, these groups are isomorphic to each other—even when associated with different minimal left ideals. We refer to these groups as  {\em ideal groups} and denote their common isomorphism class as the {\em Ellis group} of the flow $X$. For further details, readers are referred to the literature \cite{mfate, Ellis}.

\subsubsection*{Topological dynamics in model-theoretic context}
We now turn to topological dynamics in the model-theoretic setting. For an $M$-definable set $D \sq \M^n$, an {\em externally definable subset} of $D(M)$ is a set of the form $X = Y \cap D(M)$, where $Y$ is an $\M$-definable subset of $\M^n$; we write $X \subseteq_{\text{ext}} D(M)$ to indicate that $X$ is externally definable. Let $\Def^\ext(D(M))$ be the Boolean algebra of all externally definable subsets of $D(M)$, and let $S_{D,\text{ext}}(M)$ be its ultrafilter space. Passing from a fixed definable set $D$ to the whole model, write $\Def^\ext(M)$ for the Boolean algebras of externally definable subsets of $M^n$ as $n$ ranges over all positive integers, and $S_{\ext}(M)$ for the corresponding ultrafilter space, i.e. the external type space over variables of every finite arity.

Let $S_{M,D}^{\operatorname{fs}}(\M)$ denote the set of types in $S_D(\M)$ that are finitely satisfiable in $M$. There is a canonical homeomorphism, namely the trace map
\[
\operatorname{Tr} \colon S_{M,D}^{\operatorname{fs}}(\M) \to S_{D,\text{ext}}(M), \quad
p \mapsto \{\psi(\M) \cap D(M) \mid \psi \in p\}.
\]
Likewise, $S_M^{\operatorname{fs}}(\M)$ denotes the set of global types in $S(\M)$ finitely satisfiable in $M$, and the same map $p \mapsto \{\psi(\M) \cap M^n \mid \psi \in p\}$ induces a canonical homeomorphism between $S_M^{\operatorname{fs}}(\M)$ and $S_{\ext}(M)$.

In model theory, consider an $M$-definable group $G$ and   the action of $G(M)$ on the type space $S_G(M)$. Evidently, $S_G(M)$ forms a $G(M)$-flow. As shown in \cite{Newelski-TD-gp-Act}, the Ellis semigroup of $S_G(M)$ is precisely $S_{G,\ext}(M)$, with its semigroup operation explicitly characterized. We refer to $S_{G,\text{ext}}(M)$ as the {\em universal definable flow} of $G$ over $M$.

\subsubsection*{Expansion by externally definable sets}

Let $M^\ext$ be the expansion of $M$ obtained by adding a predicate for every externally definable subset of $M^n$, $n \in \mathbb{N}^{>0}$; its language $L_M^\ext$ is a natural expansion of $L$.
\begin{fact}[\cite{Shelah-NIP}]\label{Shelah-NIP-exp}
If $\Th(M)$ is NIP, that is, it does not have the Independence Property (see \cite{NIP-book} for details), then $\Th(M^\ext)$ is also NIP, admits quantifier elimination, and every type over $M^\ext$ is definable.
\end{fact}
Note that $S_{\ext}(M)$ is precisely the space of complete quantifier-free types over $M$ in the expanded language $L_M^\ext$. By Fact \ref{Shelah-NIP-exp}, in NIP theories one has $S_{\ext}(M) = S(M^\ext)$. Throughout this paper, we assume NIP and use the notation $S_G(M^\ext)$ instead of $S_{G,\ext}(M)$. 

\begin{rmk}\label{rmk:Mextext=Mext}
Note that if every complete type over a structure $M$ is definable, then every externally definable subset of any $M^n$ is already $M$-definable. In particular, for NIP theories, we have $(M^\ext)^\ext = M^\ext$.
\end{rmk}

We will use the following fact.
\begin{fact}[NIP; \cite{CPS-Ext-Def-in-NIP}]\label{fact-def-type-unique-coheir}
Let $p(x) \in S(M)$ be definable. Then $p(x)$ determines a unique complete type $p^*(x) \in S(M^{\text{ext}})$. Moreover, let $N \succ M^\ext$ be an $L_M^\ext$-structure, let $N_0 = N \upharpoonright L$ be its reduct to $L$, and let $\bar p$ be the unique heir of $p$ over $N_0$. Then $\bar p$ in turn determines a unique complete type over $N$ in the language $L_M^\ext$, which is precisely the unique heir of $p^*$.
\end{fact}

\begin{rmk}[NIP]\label{rmk-definable-type-has-a-unique-coheir}
By the canonical homeomorphism between $S(M^{\text{ext}})$ and $S_M^{\operatorname{fs}}(\M)$, Fact \ref{fact-def-type-unique-coheir} implies that a definable type $p \in S(M)$ has a unique coheir over $N$ for every $N \succ M$.
\end{rmk}

\subsubsection*{Notation for semigroup operation} 
To facilitate the definition of the semigroup operation in $S_G(M^\ext)$, we introduce a new notation: for any model $N$ and definable types $p_1(x_1), \dots, p_n(x_n) \in S(N)$, define the type  
\[
p_1(x_1) \otimes \cdots \otimes p_n(x_n) = \tp{a_1, \dots, a_n / N},
\]  
where $a_0 \models p_0$ and $a_i \models p_i | N, a_1, \dots, a_{i-1}$ for all $1 < i \leq n$.  Note that every type over $M^\ext$ is definable. By \cite{Newelski-TD-gp-Act}, the semigroup operation $*$ in $S_G(M^\ext)$  can be computed as follows: if $p,q\in S_G(M^\ext)$, $(a, b) \models p \otimes q$, then  
\[
p * q = \tp{ab / M^\ext},
\]  
where $ab$ denotes the group multiplication in $G$.



\subsection{Definable amenablity, dfg, and fsg}

\subsubsection*{Definable amenable groups}
From now on, we assume that $T$ has NIP. Let $ G $ be a group definable in $ \M \models T $. After naming the parameters, we may assume that $ G $ is 0-definable. We say that $ G $ is \emph{definably amenable} if it admits a global (left) $ G $-invariant Keisler measure, where a global Keisler measure on $ G $ is a finitely additive probabilistic measure defined on the algebra of all $ \M $-definable subsets of $ G $. 

 Recall from \cite{CS-Definably-Amenable-NIP-Groups} that a definable subset $ X $ of $G$ is (left) \emph{$ f $-generic} over a small submodel $ M $ if any left translate $ gX $ of $ X $ does not fork over $ M $. We say that $ X $ is $ f $-generic if it is $ f $-generic over some small $ M $. A (partial) type is $ f $-generic if every formula implied by it is $ f $-generic. A global type $ p $ is called (left-) \emph{strongly $ f $-generic} over $ M $ if no $ G  $-translate of $ p $ forks over $ M $. A global type $ p $ is strongly $ f $-generic if it is strongly $ f $-generic over some small model $ M $. 

Recall that a type-definable (over $ A $) subgroup $ H \leq G $ is a type-definable subset of $ G $ over $ A $ and also a subgroup of $ G $. We say that $ H $ has bounded index if $ |G/H| < 2^{|A| + |T|} $. When $ T $ has NIP, there exists a smallest type-definable subgroup of bounded index (see \cite{Shelah-min-bounded-index}), which we call the \emph{type-definable connected component} of $ G $, denoted by $ G^{00} $. We call the intersection of all definable subgroups of $G$ of finite index the \emph{definable connected component}, and denote it by $G^0$.

\begin{fact}[\cite{CPS-Ext-Def-in-NIP}]
$G^{00}$ and $G^0$ are unchanged when computed in $T$ or in $\Th(M^\ext)$. More precisely, if $N'\succ M^\ext$ is $|M|^+$-saturated and  $N=N'\restriction L$, then $G^{00}(N')=G^{00}(N)$ and $G^{0}(N')=G^{0}(N)$.
\end{fact} 

A nice result from \cite{CS-Definably-Amenable-NIP-Groups} states the following:
\begin{fact}\label{fact-definable-amenable}
 $ G $ is definably amenable if and only if there exists a strongly $ f $-generic type. Moreover, when $ G $ is definably amenable, we have:
\begin{enumerate}[label=(\roman*)]
    \item $ p \in S_G(\M) $ is $ f $-generic if and only if it is $ G^{00} $-invariant;
    \item A type-definable subgroup $ H\leq G $ of bounded index fixing a global $ f $-generic type is exactly $ G^{00} $;
    \item A global type is strongly $ f $-generic if and only if it is $ f $-generic and $ M $-invariant for some small $ M \prec \M $;
    \item For any $ N \prec \M $, if $ E \subseteq S_G(N^\text{ext}) $ is an Ellis group, then the map $ \sigma: E \rightarrow G/G^{00} $ defined by $ p \mapsto p/G^{00} $ is an isomorphism.
\end{enumerate}
\end{fact}
	
\begin{rmk}\label{rmk/def-fsg-dfg}
 Among the strongly $ f $-generics, there are two extreme cases: 
\begin{enumerate}
    \item [(i)] There exists a small submodel $ M $ such that every left $ G $-translate of $ p \in S_G(\M) $ is finitely satisfiable in $ M $. Types of this kind are referred to as \emph{fsg} (finitely satisfiable generic).
    \item [(ii)] There exists a small submodel $ M $ such that every left $ G $-translate of $ p \in S_G(\M) $ is definable over $ M $. Types of this kind are referred to as \emph{dfg} (definable $ f $-generic).
\end{enumerate}
We refer to a definable group $ G $ as \emph{fsg} (respectively, \emph{dfg}) if it admits an fsg (respectively, dfg) type. Evidently, both fsg and dfg groups are definably amenable.     
\end{rmk}

\begin{fact}[\cite{CPS-Ext-Def-in-NIP}]\label{fact-fsg-ext-fsg}
    If $G$ has fsg (respectively, dfg), then $G$ has fsg (respectively, dfg) with respect to $\Th(M^\ext)$.
\end{fact}
 We now proceed to discuss these two cases. Denote by $ \operatorname{Stab}_l(p) $ the stabilizer of $ p $ with respect to the left group action, and by $ \operatorname{Stab}_r(p) $ the stabilizer of $ p $ with respect to the right group action.

\subsubsection*{Fsg groups}
We now discuss fsg groups. We call a set $X \sq G$ \emph{generic} if finitely many left translates of $X$ cover $G$. A partial type is generic if every formula it contains is generic. A complete generic type need not exist, but when it does, Newelski proved the following.
\begin{fact}[\cite{Newelski-TD-gp-Act}]\label{fact-generic-exixts}
If $S_G(M)$ contains a generic type, then every $f$-generic type is generic, and the space of all generic types is precisely the unique minimal subflow of $S_G(M)$.
\end{fact}
We denote by $\Gen(G(M))$ the space of all generic types in $S_G(M^\ext)$; when non-empty, it is a two-sided ideal.

By \cite{gmn} we have:
\begin{fact}\label{fact-fsg}
Let $G$ be an fsg group. Then:
\begin{enumerate}
\item [(i)] If a global type $p$ is fsg, then it is finitely satisfiable in every small submodel $M$.
\item [(ii)] A global type $p$ is fsg if and only if it is generic.
\item [(iii)] $p$ is fsg if and only if $\stb[l]{p} = \stb[r]{p} = G^{00}$.
\end{enumerate}
\end{fact}

\begin{fact}[\cite{Pillay-TD-DG}]\label{fact-fsg-generic-space}
Suppose that $G$ has fsg and $M \prec \M$. Then the natural map taking a definable set $X \sq G$ to $X \cap G(M)$ induces a bijection between left (right) generic types in $S_G(\M)$ and left (right) generic types in $S_G(M^\ext)$.
\end{fact}

\begin{rmk}
Let $p \in \Gen(G(M)) \sq S_G(M^\ext)$ and let $N \succ M^\ext$ be $|M|^+$-saturated. Then $p$ has a unique generic extension over $N$. Indeed, by Remark \ref{rmk-definable-type-has-a-unique-coheir}, $p$ has a unique coheir $\bar p \in S_G(N)$. By Fact \ref{fact-fsg-ext-fsg}, $G$ has fsg in $T_M^\ext$, so by Fact \ref{fact-generic-exixts} $p$ has a generic extension $p^* \in S_G(N)$ (computed in the language $L_M^\ext$). Since $p^*$ is fsg, it is a coheir of $p$ by Fact \ref{fact-fsg}. Hence $\bar p = p^*$ is the unique generic extension of $p$.
\end{rmk}

\begin{fact}[\cite{BY-APAL}]\label{fsgms}
		Suppose that $G $ is a  fsg group. Then
		\begin{enumerate}[label=(\roman*)]
			
			\item\label{fsgms.i} For any $q \in \Gen(G(M))$, the ideal group contains $q$ is $q * \Gen(G(M)) = q * S_G(M^\ext)$.
			\item\label{fsgms.ii} For any $q \in \Gen(G(M))$ and $p_1,p_2\in S_G(M^\ext)$, we have that $q*p_1=q*p_2$ iff $p_1/G^{00}=p_2/G^{00}$.
			\item \label{fsgms.iii} For each $q\in \Gen(G(M))$ and $r\in S_G(M^\ext)$, there is $s\in q*\Gen(G(M))$ such that $q=s*r$.
		\end{enumerate}
\end{fact}

\subsubsection*{Dfg groups}
We now discuss the dfg groups. 
	
\begin{fact}[\cite{P-Y-On-minimal-flows}]\label{dfg}
Let $p \in S_G(\mathbb{M})$ be a global $f$-generic type of $G$. If $p$ is definable over $M$, then:
\begin{enumerate}[label=(\roman*)]
\item every left $G$-translate of $p$ is definable over $M$;
\item $G^{00} = G^0 = \stb[l]{p}$;
\item $G \cdot p$ is closed, and hence is a minimal subflow of $S_G(\mathbb{M})$.
\end{enumerate}
Recall that such a type $p$ is called a dfg (definable $f$-generic) type of $G$.
\end{fact}

\begin{fact}[\cite{CPS-Ext-Def-in-NIP}]\label{B(Mext)-has-dfg}
Suppose that $G$ has dfg. Let $\j \sq S_G(M^\ext)$ be a minimal subflow and let $p \in \j$. If $N^* \succ M^\ext$ is $|M|^+$-saturated, then the unique heir $\bar p \in S_G(N^*)$ of $p$ is an $f$-generic type. Moreover, every $G(N^*)$-translate of $\bar p$ is an heir of some $q \in \j$. In particular, $G$ has dfg when computed in $T_M^\ext$.
\end{fact}

\begin{fact}[\cite{BY-APAL}]\label{Mini-flow-of-B}
Suppose that  $G$ has dfg,  then every minimal subflow of $S_G(M^\ext)$is precisely an Ellis subgroup thereof.
\end{fact}

	One could conclude directly from the above Fact that 
	\begin{cor}\label{structure of J}
		If $G$ has dfg, $p\in S_G(M^\ext)$ is almost periodic, then for any $q_1,q_2\in S_G(M^\ext)$, $q_1*p=q_2*p$ iff $q_1/G^0=q_2/G^0$.
	\end{cor}

\subsection{Groups definable in $M_0$}

\subsubsection*{Dimensions}
Recall that $ M_0 $ is either the field $ \mathbb{Q}_p $ of $ p $-adic numbers or an $o$-minimal expansion of the field $ \mathbb{R} $ of real numbers. Let $ T = \Th(M_0) $. For any $ X \subseteq M_0^n $, the topological dimension, denoted $ \dim(X) $, is the largest $ m \leq n $ such that the image of $ X $ under some projection from $ M_0^n $ to $ M_0^m $ contains an open subset of $ M_0^m $. On the other hand, since the model-theoretic algebraic closure coincides with the field-theoretic algebraic closure (\cite{HP-gp-local}, Proposition 2.11), the algebraic closure operator satisfies the exchange property (thus inducing a pregeometry on models of $ T $), and there is a finite bound on the cardinalities of finite sets in uniformly definable families. For a finite tuple $ a $ from $ M \models T $ and a subset $ B \subseteq M $, the algebraic dimension of $ a $ over $ B $, denoted $ \dim(a/B) $, is the size of a maximal subtuple of $ a $ that is algebraically independent over $ B $.

For any $ M \models T $ and definable subset $ X \subseteq M^n $, the algebraic dimension of $ X $, denoted $ \alg\text{-}\dim(X) $, is the maximum $ \dim(a/B) $ where $ a \in X(\mathbb{M}) $ and $ B \subseteq M $ contains the parameters over which $ X $ is defined. Importantly, the algebraic dimension of $ X $ coincides with its topological dimension, i.e., $ \dim(X) = \alg\text{-}\dim(X) $. Consequently, for any definable $ X \subseteq M^n $, $ \dim(X) $ is exactly the algebraic geometric dimension of its Zariski closure.

\subsubsection*{Definable manifold}
Let $ M \models T $. A \emph{definable manifold} $ X \subseteq M^n $ over a subset $ A \subseteq M $ is a topological space $ X $ covered by finitely many open subsets $ U_1, \ldots, U_s $, together with homeomorphisms from each $ U_i $ to some definable open set $ V_i \subseteq M^m $ (for some $ m $), such that the transition maps are $ A $-definable and continuous. If the transition maps are $ C^l $, then $ X $ is called a definable $ C^l $-manifold over $ M $ of dimension $ m $. A \emph{definably compact} space $ X $ is one where every definable downward-directed family $ \mathcal{F} = \{ Y_z \mid z \in Z \} $ of non-empty closed sets satisfies $ \bigcap \mathcal{F} \neq \emptyset $.

Any $ M_0 $-definable group $ G $ can be uniquely equipped with the structure of a definable manifold over $ M_0 $ such that the group operation is $ C^\infty $ (see \cite{Pillay-fields-definable-over-Qp}, \cite{OP-plg}, and \cite{Pillay-g-f-definable-ino-min}). As noted in \cite{OP-plg}:

\begin{fact}\label{fact-open-compact-subgp}
If $ G $ is definable over $ \Q $, then $ G $ has a $ \Q $-definable open subgroup $ C $ that is definably compact (equivalently, $ C(\Q) $ is compact).
\end{fact}

The following fact was established in \cite{OP-plg} for $ G $ being $ \Q $-definable and definably compact, and extended in \cite{JY-3} to arbitrary $ \Q $-definable groups:

\begin{fact}
If $ G $ is definable over $ \Q $, then $ G^0 = G^{00} $.
\end{fact}

\subsubsection*{Fsg groups over $ M_0 $}
A group $G \subseteq \mathbb{M}^m$ definable over $M_0$ has fsg if and only if it is definably compact. Equivalently, $G(M_0)$ is compact. The standard part map $\operatorname{st}: G \to G(M_0)$ is well-defined, and the type-definable connected component $G^{00}$ coincides with its kernel; that is, $G^{00}$ is precisely the set of infinitesimals of $G$ over $M_0$. See \cite{gmn, PP-Gen-set, OP-plg} for further details.

Moreover, being definably compact is definable in $T=\Th(M_0)$. That is, for $\{X_b\mid b\in Y\}$, a definable family of definable manifolds, the set $\{b\in Y\mid X_b \text{ is definably compact}\}$ is definable. In particular, for $\{G_b\mid b\in Y\}$, a definable family of definable groups, the set $\{b\in Y\mid G_b \text{ has fsg}\}$ is definable.

\subsubsection*{Dfg groups over $ M_0 $}

We now discuss dfg groups definable over $ M_0 $:

\begin{fact}[\cite{PY-dfg}]\label{fact-dfg-group-Qp}
Let $ G $ be a group definable over $ \mathbb{Q}_p $. The following are equivalent:
\begin{enumerate}
    \item $ G $ has dfg.
    \item There exists a trigonalizable algebraic group $ A $ over $ M_0 $ and a definable homomorphism $ f : H \to A(\M) $ such that both $ \ker(f) $ and $ A(\M)/\mathrm{im}(f) $ are finite.
    \item There exist definable subgroups
    \[
    G_0 \lhd \cdots \lhd G_n = G
    \]
    where $ G_0 $ is finite, and each $ G_{i+1}/G_i $ is definably isomorphic to the additive group or a finite-index subgroup of the multiplicative group.
\end{enumerate}
\end{fact}

The following fact follows from Proposition 2.1 and Proposition 4.1 in \cite{CP-connected}, together with Fact \ref{fact-main-conj-RCF} below (see Remark \ref{rmk-dfg=tf}):

\begin{fact}[\cite{CP-connected}]\label{fact-dfg-group-R}
Let $ G $ be a definably connected  (i.e., $ G = G^0 $) group definable in an $o$-minimal expansion of a real closed field. The following are equivalent:
\begin{enumerate}
    \item $ G $ has dfg.
    \item $ G $ is torsion-free.
    \item There exist definable subgroups
    \[
    \{1\} = G_0 \lhd \cdots \lhd G_n = G
    \]
    such that for each $ i < n $, $ G_{i+1}/G_i $ is 1-dimensional and torsion-free.
\end{enumerate}
\end{fact}

\begin{rmk}

Note that any 1-dimensional group $ G $ definable over $ M_0 $ is commutative-by-finite (see \cite{V. Razenj, PY-A-note}) and thus definably amenable. Take a global strong $ f $-generic type $ p $ over $ M_0 $; $ p $ is then $ M_0 $-invariant. Since $ p $ is dp-minimal, by \cite{Simon-dp-minimal}, it is either finitely satisfiable in $ M_0 $ or definable over $ M_0 $. It follows that any such 1-dimensional group has either dfg or fsg. Furthermore, since $ \Th(M_0) $ is a distal theory, such a group cannot have both dfg and fsg.

Thus, by Facts \ref{fact-dfg-group-Qp} and \ref{fact-dfg-group-R}, $ G $ has dfg if and only if it is ``solvable with no definably compact parts". 
\end{rmk}


Moreover, having dfg is definable in $T = \Th(M_0)$. This was proved in \cite{PYZ-$p$-adic-groups} for $M_0 = \mathbb{Q}_p$, and the same proof goes through when $M_0$ is an $o$-minimal expansion of $\mathbb{R}$. More explicitly, for any definable family $\{G_b \mid b \in Y\}$ of definable groups, the set
\[
\{b \in Y \mid G_b \text{ has dfg}\}
\]
is definable.


In addition, the quotient space arising from a definable action of a dfg group is definable (and not merely interpretable). For $M_0 = \mathbb{Q}_p$, this is known as Corollary 4.3 of \cite{GJ-Around}. When $M_0$ is an $o$-minimal expansion of $\mathbb{R}$, it follows directly from elimination of imaginaries in $T$.
\begin{fact}\label{fact-G/dfg-is-definable}
Let $H$ be a dfg group acting definably on a definable set $X$. Then the quotient $X/H$ is definable (and not merely interpretable). In particular, if $H$ is a dfg subgroup of a definable group $G$, then $G/H$ is definable.
\end{fact}


\subsubsection*{Dfg/fsg decompositions for groups over $ M_0 $}

\begin{dfn}\label{def-dfg component}
Let $ H $ be a dfg subgroup of $ G $. We call $ H $ a \emph{dfg component} of $ G $ if $ \dim(H) $ is maximal among all dfg subgroups of $ G $.
\end{dfn}
\begin{fact}\label{fact-main-conj-RCF}
Let $ G $ be a group definable over an $o$-minimal expansion of a real closed field. Then:
\begin{enumerate}
    \item A dfg subgroup $ H $ of $ G $ is a dfg component of $ G $ if and only if $ G/H $ is definably compact (see \cite{Con-A-Reduc}, Corollary 2.4);
    \item The dfg components of $ G $ are conjugate-commensurable: If $ H_1 $ and $ H_2 $ are dfg components of $ G $, then there exists $ g \in G $ such that $ H_1^g \cap H_2 $ has finite index in $ H_2 $ (see \cite{PS-mu-stablizer}, Theorem 3.26);
    \item $ G $ is definably amenable if and only if $ G $ has a normal dfg component (see \cite{CP-connected}, Proposition 4.6).
\end{enumerate}
\end{fact}


Note that when $ G $ is definably connected (i.e., $ G = G^0 $), a dimension-maximal torsion-free subgroup is also a maximal torsion-free subgroup.

\begin{rmk}\label{rmk-dfg=tf}
Let $ G $ be a definably connected group definable over an $o$-minimal expansion of a real closed field. From Proposition 2.1 and the proof of Proposition 4.7 in \cite{CP-connected}, it follows easily that if $ G $ is torsion-free, then $ G $ has dfg. Conversely, if $ G $ has dfg, then $ G $ is definably amenable and thus has a torsion-free subgroup $ H $ such that $ C = G/H $ is definably compact. Let $ \pi : G \to C $ be the projection, and let $ p \in S_G(\mathbb{M}) $ be a global dfg type. Then $ q = \pi(p) \in S_C(\mathbb{M}) $ is a global dfg type of $ C $, so $ C $ has both dfg and fsg. This implies $ C $ is trivial, so $ G $ is torsion-free.
\end{rmk}




Recent results from our collaboration with A. Pillay in \cite{PYZ-$p$-adic-groups} show that Fact \ref{fact-main-conj-RCF} also holds for $ p $-adically closed fields:

\begin{fact}[\cite{PYZ-$p$-adic-groups}]\label{fact-main-conj-pCF}
Let $ G $ be a group definable over a $ p $-adically closed field. Then:
\begin{enumerate}
    \item A dfg subgroup $ H $ of $ G $ is a dfg component if and only if $ G/H $ is definably compact;
    \item The  dfg conponents of $G$ are conjugate-commensurable: If $ H_1 $ and $ H_2 $ are dfg conponents of $ G $, then there exists $ g \in G $ such that $ H_1^g \cap H_2 $ has finite index in $ H_2 $;
    \item $ G $ is definably amenable if and only if $ G $ has a normal dfg conponent.
\end{enumerate}
\end{fact}


The dfg-fsg decomposition for definably amenable groups implies that definable amenability is a definable property in $T=\Th(M_0)$.

\begin{pro}\label{prop-DA-def}
Let $\{G_b \mid b\in Y\}$ be a definable family of definable groups (over $\emptyset$, say). Then the set $Z:=\{b\in Y \mid G_b \text{ is definably amenable}\}$ is definable over $\emptyset$.
\end{pro}
\begin{proof}
We first show that $Z$ is $\vee$-definable. For each formula $\psi(x;z)\in L$, since both having dfg and definable compactness are definable properties, there exists a formula $\hat{\psi}(y,z)$  expressing that $H_{z}:=\psi(\M;z)$ is a definable normal dfg subgroup of $G_y$ and $G_y/H_z$ is definably compact. Then by the dfg-fsg decomposition for definably amenable groups, $b\in Z$ iff $\exists z \hat{\psi}(b,z)$ for some $\psi\in L$. Thus, $Z$ is $\vee$-definable.

Now we show, without any assumptions on the theory, that $Z$ is type-definable. Fix formulas $\varphi_1(x;u),\dots, \varphi_n(x;u)$ and $m\in \N^{>0}$. Let $I_{n,m}=\{1,\dots,n\}\times\{1,\dots, m\}$. Given $b\in Y$, $c$ for the variable $u$, and $g_1,\dots,g_m\in G_b$, we let $X_i:=G_b\cap \varphi_i(\M,c)$. For $\eta\in\{0,1\}^{I_{n,m}}$, we let $P_\eta\in L$ such that 
$$P_\eta(\M;b,c,g_1,\dots,g_m):=\bigcap_{\eta(i,j)=1} g_j X_i\cap \bigcap_{\eta(i,j)=0}(G_b\backslash  g_jX_i).$$
These sets form a partition of $G_b$, with some parts possibly empty.

For convenience, we say that $\Lambda\subseteq\{0,1\}^{I_{n,m}}$ is admissible if the following finite system of linear equations in the real variables $(r_\eta)_{\eta\in\Lambda}$ has a solution in $[0,1]$:
$$\sum_{\eta\in \Lambda} r_\eta=1,\qquad
\sum_{\eta\in \Lambda,\, \eta(i,j_1)=1} r_\eta=\sum_{\eta\in\Lambda,\, \eta(i,j_2)=1} r_\eta
\quad\text{for all }1\leq i\leq n,\ 1\leq j_1,j_2\leq m.$$
It is easy to check that an admissible $\Lambda$ exists iff there exists a finitely additive probability measure $\nu$ on the finite Boolean algebra generated by $\{g_jX_i\mid (i,j)\in I_{n,m}\}$ such that $\nu(g_{j_1} X_i)=\nu(g_{j_2}X_i)$ for all $i,j_1,j_2$.

We let $v=(v_1,\dots,v_m)$ and $$\theta_{\Lambda,m,\varphi_1,\dots,\varphi_n}(y,u,v):=\bigwedge_{\eta\in \Lambda} \exists x P_\eta(x,y,u,v)\wedge \bigwedge_{\eta\notin \Lambda} \neg\exists x P_\eta(x,y,u,v).$$
Obviously,  $\theta(b,u,g_1,\dots,g_m)$ says that $\Lambda$ are exactly the collection of $\eta$ such that  $P_\eta(\M;b,u,g_1,\dots,g_n)$ is nonempty. Then we let
$$\sigma_{m,\varphi_1,\dots,\varphi_n}(y,u,v):=\bigvee_{\Lambda  \text{ is admissible}}\theta_{\Lambda,m,\varphi_1,\dots,\varphi_n}(y,u,v).$$
So $\sigma_{m,\varphi_1,\dots,\varphi_n}(b,c,g_1,\dots,g_m)$ says that  there exists a finitely additive probability measure $\nu$ on the  finite Boolean algebra generated
by $g_jX_i$ such that $\nu(g_{j_1} X_i)=\nu(g_{j_2}X_i)$.

We let $$\chi_{m,\varphi_1,\dots,\varphi_n}(y):=y\in Y\wedge\forall u\forall v((\bigwedge_{i=1}^m v_i\in G_y)\rightarrow \sigma_{m,\varphi_1,\dots,\varphi_n}(y,u,v)).$$
Note that, for $b\in Y$, $[0,1]^{\Def(G_b)}$ is compact where $\Def(G_b)$ is the Boolean algebra of all definable subsets
of $G_b$. Then, by compactness, it is easy to see that $G_b$ is definably amenable iff $\chi_{m,\varphi_1,\dots,\varphi_n}(b)$ holds for every $m\in \N^{>0}$ and every finite sequence of formulas $\varphi_1,\dots, \varphi_n$. Thus, $Z$ is type-definable.

Since $Z$ is both $\vee$-definable and type-definable, $Z$ is definable.  Clearly, 
$Z$ is definable over the same parameters as the definable family.
\end{proof}


\subsection{Algebraic groups}
Let $k$ be any filed with characteriatic $0$,  by an algebraic group $G$ over a field $k$, we mean a group object in the category of algebraic varieties over $k$. Let $\Omega$ be a very saturated algebraically closed field (of course contains $k$), we may identify an algebraic group $G$ with the group $G(\Omega)$ of its $\Omega$-points. We can consider the group $G(k)$ of $k$-rational points of $G$. $G(k)$ will in particular be a group (quantifier-free) definable in the field $(k,+,\times)$.  We may sometimes refer to $G(k)$ as a ``$k$-algebraic group".   For our purposes we may assume that all   algebraic groups in  this paper are connected.

Let $ G $ be an algebraic group over $ k $. The solvable radical, denoted by $ R(G) $, of $ G $  the maximal connected normal solvable algebraic subgroup of $ G $, which is also defined over $ k $.  We say that $ G $ is \underline{semi-simple} if  $ R(G) $ is trivial. Equivalently,  $ G $ is semi-simple if it contains no nontrivial connected commutative normal algebraic subgroup. We say $ G $ is \underline{$ k $-simple} if it is non-commutative and has no proper normal algebraic subgroups over $ k $ other than $ \{\id_G\} $, and \underline{almost $ k$-simple} if it is non-commutative and has no proper normal algebraic subgroups over $ k $ except for finite subgroups. By the theorem in \S27.5 of Humphreys \cite{Humphreys}, we have:
\begin{fact}\label{fact-connected-semi-simples}
Let $G$ be a connected algebraic group over $k$. Then $G$ is semisimple if and only if it is an almost direct product of almost $k$-simple algebraic groups over $k$. Moreover, when $Z(G)=\{e\}$, this almost direct product is a genuine direct product of $k$-simple algebraic groups over $k$.
\end{fact}


\begin{notations}
 For the remainder of the paper, we fix a structure $M$ which is either a $p$-adically closed field ($\pcf$) or an $o$-minimal expansion of a real closed field ($\mathrm{RCF}$). We let $\M \succ M$ be a monster model of $\mathrm{Th}(M)$, with underlying field $\mathbb{F}$. We denote by $\Omega$ the algebraic closure of $\mathbb{F}$, which serves as our universal domain for all algebraic geometric constructions. 

Given a subset $A$ of $M^n$, We write \(\clz{A}\) for the Zariski closure of a set $A$ in \(\Omega\).   We use $\Zdim(X)$ to denote the Zariski dimension of an algebraic variety $X$ in $\Omega$.

A definable group $H$ in $\M$ is said to be \emph{linear} if it is a definable subgroup of $\mathrm{GL}_n(\M)$ for some $n$.   Clearly, $\clz{H}$ is a defined over $M$ if $H$ is definable over $M$, and $H \leq \clz{H}(\M)$. A linear group $H$ is said to be \emph{semisimple} (resp. \emph{almost simple over $M$}) if $\clz{H}$ is a semisimple (resp. almost simple) algebraic group defined over $M$.
\end{notations}

\begin{fact}[\cite{HP-gp-local}]
\begin{enumerate}
    \item [(i)] Let $X \sq \M^n$ be a definable set. Then $\dim(X) = \Zdim(\clz{X})$. Consequently, if $H \sq \M^n$ is a linear group, then $H$ is open in $\clz{H}(\M)$.
    \item [(ii)] If $G$ is an algebraic group over $M$, then $\dim(G(M)) = \dim(G)$.
\end{enumerate}

\end{fact}

\begin{fact}[Lemma 5.5 of \cite{PYZ-$p$-adic-groups} and Corollary 2.22 of \cite{MOS}]\label{fact-dfg-alg-group}
Let $G$ be a connected algebraic group over $M$. Then $G(\M)$ has dfg iff $G$ is $M$-split solvable, i.e.\ $G$ is isomorphic to a semidirect product of an $M$-split torus by a unipotent algebraic group over $M$.
\end{fact}
\begin{rmk}\label{rmk-unip-split}
If $G$ is a unipotent algebraic group over $M$, then $G$ is $M$-split, i.e.\  G admits a composition series whose successive quotients are $M$-isomorphic to \(\mathbb{G}_a\) (see \cite{Borel-book} Corollary 15.5). Since $\mathbb{G}_a(\M)=(\M,+)$ is definably connected, $G(\M)$ is definably connected.
\end{rmk}

\begin{lemma}\label{lemma-almost-k-simple-M-simple}
    Let $ G $ be a connected algebraic group over $M$. If $ G $ is   almost $ M $-simple, then $ G $ is  almost $ \mathbb{M} $-simple.
\end{lemma}
\begin{proof}
    Suppose $G$ is  almost $M$-simple, but not almost $ \mathbb{M} $-simple. Then there exists a quantifier-free $ L_{\text{ring}} $-formula $ \psi(x,y) $ and $ b \in \mathbb{M} $ such that $ \psi(\Omega,b) $ is a   normal algebraic subgroup of $ G $ with 
    \[
    0<\Zdim(\psi(\Omega,b))<\Zdim(G) .
    \]
    It follows that $ \psi(\mathbb{M},b) $ is a normal definable subgroup of $ G(\mathbb{M}) $ with \[
    0<\dim(\psi(\M,b))<\dim(G(\M)).\]
    Since the statement ``$ \dim(\psi(\M,y)) = n $'' is expressible by a first-order formula, we have
    \[
    \mathbb{M} \models \exists y \left( \psi(\mathbb{M},y) \text{ is a  normal subgroup of } G(\mathbb{M}) \text{ and }0<\dim(\psi(\M,b))<\dim(G(\M)) \right).
    \]
    Therefore, there exists $ b_0 \in M $ such that $ \psi(\mathbb{M},b_0) $ is a normal subgroup of $ G(\mathbb{M}) $ with $ 0<\dim(\psi(\M,b))<\dim(G(\M))$. The Zariski closure $ \clz{\psi(\mathbb{M},b_0)} $ of $ \psi(\mathbb{M},b_0) $ is thus an infinite normal subgroup of $ G $ with lower dimension, defined over $ M $. A contradiction.
\end{proof}

\section{Main results}


\subsection{The definably-amenable--semisimple decomposition}




\begin{fact}[Claim 2.26 of \cite{PPS-TAMS} and Lemma 2.11 of \cite{JY-3}]\label{open-subgp-of-semi-simple}
 Let $ G $ be a group definable over $ M $.   There is a definable short exact sequence of $M$-definable groups
\begin{equation}\label{eq-2}
1 \to Z \to G \stackrel{\pi_L}{\to} L \to 1,    
\end{equation}
where $Z$ is commutative-by-finite, and $L$ is linear.
\end{fact}

We call a definable group $G$   {\em definably simple} if it has no proper normal definable subgroup expect for $\{\id_G\}$.

\begin{lemma}\label{claim-simple-linear}
Let $ G$ be a group definable over $M$. Suppose that $G$ is definably simple, infinite, and noncommutative, then  $G$ is linear and the Zariski closure $\clz{G}$ of $G $ is almost $M$-simple.
\end{lemma}
\begin{proof}
  By Fact \ref{open-subgp-of-semi-simple}, $G$ is linear.    Suppose for a contradiction that $X$ is a connected infinite proper normal algebraic subgroup of $\clz{G}$ defined over $M$. As $G$ is definably simple, $G \cap X(\M)$ is a proper normal subgroup of $G$ so trivial. Hence $G$ is naturally a subgroup of $\clz{G}/X$, which is also a connected algebraic group over $M$. But $\Zdim(\clz{G}/X)$  is strictly less than  $\Zdim(\clz{G})$, so   $\dim(G)\leq \Zdim(\clz{G}/X)< \dim(G)$.  This is a contradiction.
\end{proof}

Let $L\leq \GL_n(\M)$ be a linear group. Then $L$ is an open subgroup of $\clz{L}(\M)$. Let $\clz{R}$ be solvable radical of $\clz{L}$, then $\clz{S}=\clz{L}/\clz{R}$ is simple-simple. Now the short exact sequence of algebraic groups
\[
1\to \clz{R}\to \clz{L}\to \clz{S}\to 1
\]
induces a definable short exact sequence
\begin{equation}\label{eq-1}
1\to R \to L\to S\to 1,
\end{equation}
where $R=\clz{R}(\M)\cap L$ is solvable and $S$ is a definable open subgroup of $\clz{S}(\M)$, thus is semi-simple.

\begin{lemma}\label{lem-amenable-semi-simple-decom}
Let $G$ be a group definable over $M$. If $G$ is not definably amenable, then there exists an $M$-definable short exact sequence
\[
1 \to D \to G \stackrel{\pi_A}{\to} A \to 1,
\]
where $D$ is definably amenable, $A$ is linear, $\clz{A}$ is centreless and decomposes as a direct product of $M$-simple algebraic groups $\mathbf{A}_1,\dots,\mathbf{A}_m$, and for each $i=1,\dots,m$, the subgroup $A\cap \mathbf{A}_i(\M)$ is not definably compact.
\end{lemma}

\begin{proof}
Consider the definable short exact sequence given by \eqref{eq-2} in Fact \ref{open-subgp-of-semi-simple}:
\[
1 \to Z \to G \stackrel{\pi_L}{\to} L \to 1,
\]
together with the short exact sequence given by \eqref{eq-1},
\[
1 \to R \to L \to S \to 1,
\]
where $R$ is solvable and $S$ is semisimple. Quotienting $S$ by its finite center and pulling this quotient back to $G$, we may assume that $S$ is centreless. We then obtain
\[
1 \to \pi_L^{-1}(R) \to G \to G/\pi_L^{-1}(R) \to 1.
\]
Set $D_0=\pi_L^{-1}(R)$. Since $D_0$ is an extension of the solvable group $R$ by $Z$, and both $R$ and $Z$ are definably amenable, $D_0$ is definably amenable. Moreover,
\[
B:=G/D_0 \cong L/R \cong S
\]
is semisimple and centreless. Hence $\mathbf{B}=\clz{B}$ is a centreless semisimple algebraic group, and therefore, by Fact \ref{fact-connected-semi-simples}, it decomposes as a direct product of $M$-simple algebraic groups $\mathbf{B}_1,\dots,\mathbf{B}_m$.

Suppose that $ \mathbf{B}_i(\M)\cap B$ is definably compact for some $i$. Then there is an open subgroup $S_i$ of $(\mathbf{B}/\mathbf{B}_i)(\M)$ fitting into a short exact sequence
\[
1 \to  \mathbf{B}_i(\M)\cap B \to B \to S_i \to 1.
\]
Set $D_1=\pi_B^{-1}( \mathbf{B}_i(\M)\cap B)$. The group $ \mathbf{B}_i(\M)\cap B$ is definably compact, hence definably amenable; as an extension of $D_0$ by this group, $D_1$ is again definably amenable. The quotient $G/D_1$ is linear, and its Zariski closure is definably isomorphic over $M$ to a direct product of $M$-simple algebraic groups. Iterating this quotient for every definably compact factor, we eventually obtain a subgroup $D$ and a quotient $A=G/D$ such that $D$ is definably amenable, $A$ is linear, and $\clz{A}$ is a direct product of $M$-simple groups none of whose $M$-points is definably compact.

Finally, $A$ is nontrivial: otherwise $G=D$ would be definably amenable, contrary to the assumption. This completes the proof.
\end{proof}

\begin{dfn}
 Let $G$ be a semisimple algebraic group over a field $k$. By $G(k)^+$ we mean the (abstract) subgroup of $G(k)$ generated by its unipotent elements.   
\end{dfn}
Recall that an algebraic group $ G $ over $k$ is said to be \emph{$ k $-isotropic} if it admits a $ k $-split algebraic torus as a $ k $-subgroup—here, a $ k $-split algebraic torus over $ k $ refers to an algebraic subgroup over $ k $ that is $ k $-isomorphic to a finite direct power of the multiplicative group. An algebraic group is \emph{$ k $-anisotropic} if it is not $ k $-isotropic. We now view $\GL_n(\Omega)$ as a closed subset of $\Omega^{n^2+1}$.

\begin{fact}[\cite{Prasad}]\label{fact-semi-simple-unbounded=isotropic}
Let $ k$ be a henselian valued field with valuation ring $ \mathcal{O}_k $, and let $ G \subseteq \operatorname{GL}_n(\Omega) $ be a reductive algebraic group over $ k $. Then $ G $ is $k $-anisotropic if and only if $ G(k) $ is bounded (i.e., there exists $ \alpha \in k $ such that $ G(k) \subseteq \alpha \mathcal{O}_k^{n^2 + 1} $) in $ \operatorname{GL}_n(k) $.
\end{fact}

\begin{rmk}\label{rmk-aniso=def-comp}
 Let $G$ be an algebraic group over $M$. It follows readily from Fact \ref{fact-semi-simple-unbounded=isotropic} that for $M\models \pcf$, $G$ is $M$-anisotropic if and only if $G(M)$ is definably compact. On the other hand, for $M$ an $o$-minimal expansion of $\mathrm{RCF}$, as noted in Remark 6.2 of \cite{PPS-JA}, we likewise have that $G$ is $M$-anisotropic if and only if $G(M)$ is definably compact.
\end{rmk}

\begin{fact}[See {\cite[Fact 6.6]{PPS-JA}}]\label{fact-G+-has no inf nor sbgp}
Let $G$ be an algebraic group over a field $k$ of characteristic $0$. Suppose $G$ is almost $k$-simple and $k$-isotropic. Then $G(k)^+$ has no proper infinite normal subgroups.
\end{fact}

Recall that the Kneser-Tits conjecture is:

\begin{KTC}[See {\cite[Chapter 7]{AG-and-NT}}]
Let $G$ be a simply connected, almost $k$-simple, $k$-isotropic algebraic group over a field $k$. Then $G(k)^+ = G(k)$.
\end{KTC}

The Kneser-Tits Conjecture is true for $ k = \mathbb{R} $ and $ k = \mathbb{Q}_p $ (see Proposition 7.6 and Theorem 7.6 in \cite{AG-and-NT}); it has further been proven to hold for any real closed field (see Proposition 6.8 in \cite{PPS-JA}) as well as for any $ p $-adically closed field (see Proposition 3.13 in \cite{PYZ-$p$-adic-groups}).


In Corollary 3.15 of \cite{PYZ-$p$-adic-groups}, the authors only proved the following fact in the case where $M$ is a $p$-adically closed field, but the same proof applies to $\mathrm{RCF}$ where the Kneser-Tits Conjecture also holds.

\begin{fact}\label{fact-G(k)+}
Let  $ G $ be an almost $M$-simple, $M$-isotropic algebraic group over $M$. Then
\begin{enumerate}
    \item[\upshape(i)] $ G(M)^+ $ is definable, normal, and of finite index;
    \item[\upshape(ii)] Any open subgroup $ H $ of $ G(M) $ is either definably compact (this case does not occur when $M$ is a real closed field), or contains $ G(M)^+ $ (and is therefore definable and of finite index).
\end{enumerate}
\end{fact}


 Let $G$ be an almost $M$-simple, $M$-isotropic algebraic group over $M$. For $M$ an $o$-minimal expansion of $\mathrm{RCF}$, Lemma 4.5 of \cite{CP-connected} shows that $G(M)$ is not definably amenable; for $M\models\pcf$, the same conclusion holds by Lemma 4.3 of \cite{PYZ-$p$-adic-groups}. Combining this with Remark \ref{rmk-aniso=def-comp} (which covers the $M$-anisotropic case, where $G(M)$ is definably compact) and Fact \ref{fact-G(k)+}, we derive the following characterization for all almost $M$-simple algebraic groups:

\begin{fact}\label{fact:DA-iff-compact-almost-simple}
 Let $G$ be an almost $M$-simple algebraic group over $M$. An open subgroup $X$ of $G(M)$ is definably amenable if and only if it is definably compact. Consequently, when $G$ is $M$-isotropic, such an $X$ is definably amenable if and only if it has infinite index in $G(M)$.
\end{fact}

Applying Lemma \ref{lemma-almost-k-simple-M-simple}, Lemma \ref{lem-amenable-semi-simple-decom},    Fact \ref{fact-G(k)+}, and Fact \ref{fact:DA-iff-compact-almost-simple}, it is straightforward to show that
\begin{cor}\label{cor-amenable-semi-simple-decom}
Let $G$ be a group definable over $M$. If $G$ is not definably amenable, then there exists an $M$-definable short exact sequence
\begin{equation}\label{equ-amenable-semisimple-I}
1 \to D \to G \stackrel{\pi_A}{\to} A \to 1,
\end{equation}
where $ D $ is definably amenable, $A$ is linear, and $\mathbf{A}=\clz{A}$ is a direct product of $M$-simple algebraic groups $ \mathbf{A}_1, \ldots, \mathbf{A}_m $ such that each $ \mathbf{A}_i(\M) $ is not definably compact, $ A $ has finite index in $ \mathbf{A}(\M) $. 
\end{cor}

\begin{dfn}
Let $G$ be a  group definable over $M$.
 We call the exact sequence (\ref{equ-amenable-semisimple-I}) of Corollary \ref{cor-amenable-semi-simple-decom} a \emph{definably-amenable--semisimple decomposition} of the definable group $G$ when $G$ is not definably amenable,. We additionally let $D=G$ when $G$ is definably amenable. 
\end{dfn}



\begin{lemma}\label{lemma-Y00 has finite index}
Let $G $ be an algebraic group over $M$, and let $A$ be a non-definably-compact definable open subgroup of $G(\M)$. If $G$ is almost $M$-simple, then $A^{00}=A^0$ and this subgroup has finite index in $G(\M)$.
\end{lemma}

\begin{proof}
By Lemma \ref{lemma-almost-k-simple-M-simple}, $G$ is almost $\mathbb{M}$-simple. Fact \ref{fact-G(k)+} then implies that $G(\mathbb{M})^+\leq A \leq G(\mathbb{M})$. Furthermore, every finite-index subgroup of $A$ contains $G(\mathbb{M})^+$. We thus conclude that $A^0=G(\mathbb{M})^+$, which is a finite-index subgroup of $G(\mathbb{M})$.

Next, $A^{00}$ is an infinite normal subgroup of $G(\mathbb{M})^+$. By Fact \ref{fact-G+-has no inf nor sbgp}, it follows that $A^{00}=G(\mathbb{M})^+$. Hence, $A^{00}=A^0=G(\mathbb{M})^+$, and this subgroup has finite index in $G(\M)$.
\end{proof}

An immediate application of Lemma \ref{lemma-Y00 has finite index} yields:
\begin{cor}
Let $A$ be as in the exact sequence (\ref{equ-amenable-semisimple-I}) of Corollary \ref{cor-amenable-semi-simple-decom}. Then $A^{00} = A^0$, which is of finite index in $A$.
\end{cor}

\begin{rmk}
Let $A$ be as in the exact sequence (\ref{equ-amenable-semisimple-I}) of Corollary \ref{cor-amenable-semi-simple-decom}. For $M$ an $o$-minimal expansion of a real closed field, it is shown in \cite{PPS-JA} that $A^0$ is simple as an abstract group; for $M$ a $p$-adically closed field, the corresponding result holds by \cite{PYZ-$p$-adic-groups}.
\end{rmk}


In fact, we may assume that the group $A$ in (\ref{equ-amenable-semisimple-I}) of Corollary \ref{cor-amenable-semi-simple-decom} is defined over the standard model $M_0$. The following fact is precisely Corollary 3.7 of \cite{PYZ-$p$-adic-groups} in the $p$-adic case; the same argument yields the $o$-minimal case, after replacing the $p$-adic rationality input by its real analogue, Proposition 2.4 of \cite{Witte}.
\begin{fact}\label{fact-almost M-simple-Q-simple}
 If $A$ is an almost $M$-simple algebraic group over $M$, then $A$ is $M$-isomorphic to an almost $\mathbb{Q}$-simple algebraic group over $\mathbb{Q}$.
\end{fact}


\begin{pro}\label{prop-standard-A}
The group $A$ in (\ref{equ-amenable-semisimple-I}) can be chosen to be definable over $M_0$.
\end{pro}
\begin{proof}
Recall that $A$ is a finite-index subgroup of $\mathbf{A}(\M)$, where $\mathbf{A}$ is a direct product of almost $M$-simple algebraic groups $\mathbf{A}_i$. By Fact \ref{fact-almost M-simple-Q-simple}, each $\mathbf{A}_i$ may be taken to be definable over $M_0$, and hence so may $\mathbf{A}$. In this case, since $\mathbf{A}(\M)^0$ is of finite index in $\mathbf{A}(\M)$ and is $M_0$-definable by Lemma \ref{lemma-Y00 has finite index}, it follows that $A$ is definable over $M_0$.
\end{proof}


 
Moreover, the definably-amenable--semisimple decomposition is unique:
\begin{pro}\label{pro-definably amenable radical}
We maintain the notation of Corollary \ref{cor-amenable-semi-simple-decom}. Then $D$ is the maximal definably amenable normal subgroup of $G$. Recall, following the terminology of \cite{amenable}, that we call such $D$ the {\em definably amenable radical} of $G$.
\end{pro}

\begin{proof}
Let $D'$ be a definably amenable normal subgroup of $G$, and set $E=\pi_A(D')$. Since $\pi_A$ is surjective and $D'\trianglelefteq G$, $E$ is a normal subgroup of $A$, and as a homomorphic image of $D'$ it is definably amenable.

Now $\mathbf E=\clz{E}$ is a normal algebraic subgroup of $\mathbf A=\clz{A}$. By Lemma \ref{lem-amenable-semi-simple-decom}, $\mathbf A$ is a direct product of $M$-simple algebraic groups $\mathbf{A}_1,\dots,\mathbf{A}_m$, so every normal algebraic subgroup of $\mathbf A$ is a product of some of these factors. If $\mathbf E$ were nontrivial, it would contain such an $M$-simple factor, say $\mathbf{A}_1$. Let $\pi_1:\mathbf A\to\mathbf A_1$ be the corresponding projection. Then $\pi_1(E)$ is an open, normal, definably amenable subgroup of $\mathbf{A}_1(\M)$. By Fact \ref{fact:DA-iff-compact-almost-simple}, it has infinite index in $\mathbf{A}_1(\M)$. By Fact \ref{fact-G(k)+}, $\pi_1(E)\cap \mathbf{A}_1(\M)^+$ is open in $\mathbf{A}_1(\M)^+$ and still has infinite index there, contradicting Fact \ref{fact-G+-has no inf nor sbgp}.

Therefore $\mathbf E=\{e\}$, so $E=\{e\}$, and $D'\subseteq \ker\pi_A=D$.
\end{proof}

\subsection{The definably amenable component of $G$}\label{sec-Amenable Component}




\begin{lemma}\label{lemma-open-is-finite-index}
    Let $H$ be a dfg subgroup of $G$ and $A$ a definable subgroup of $G$. If $H\cap A$ is open in $A$, then $H\cap A$ has finite index in $A$. In particular, if $H$ is open in $G$, then $H$ has finite index in $G$.
\end{lemma}
\begin{proof}
By Fact \ref{fact-G/dfg-is-definable}, $ A/H = \{ a/H \mid a \in A \} $ is definable. Since the map $ a/(H \cap A) \mapsto a/H $ is an interpretable bijection from $ A/(H \cap A) $ to $ A/H $, we see that $ A/(H\cap A) $ is also definable. Since $ H \cap A $ is open in $ A $, we have $ \dim(A) = \dim(H \cap A) $. Therefore,  
\begin{equation}\label{{equa-dim}}
\dim(A/(H \cap A)) = \dim(A) - \dim(H \cap A) = 0 .  
\end{equation}
So $ A/(H \cap A) $ is finite as required.  
\end{proof}

\begin{lemma}
Let $ H  $ be a dfg subgroup of $G$. Then $ N_G(H^0) = \{ g \in G \mid ({H^0})^g = H^0 \} $ is definable.
\end{lemma}


\begin{proof}
This is trivial in the case where $ M $ is a real closed field, as $ H^0 $ is definable. When $ M $ is a $ p $-adically closed field, this fact is exactly Lemma 9.1 of \cite{PYZ-$p$-adic-groups}.
\end{proof}

Let $H$ be an $M$-definable  dfg component of $G$. The above Lemma shows that  $ N_G(H^0)$ is definable.  Let ${\cal H}=\bigcap_{b\in N_G(H^0)}H^b$, then ${\cal H}$ is clearly a $M$-definable subgroup of $H$. Since $H^0\leq {\cal H}$, we see that ${\cal H}$ has finite index in $H$.

\begin{lemma}
$N_G(H^0) = N_G({\cal H})$.
\end{lemma}

\begin{proof}
This proof is taken from Lemma 9.2 of \cite{PYZ-$p$-adic-groups}: If $ g \in N_G({\cal H}) $, then $ h \mapsto h^g $ is a definable automorphism of $ {\cal H} $, so it fixes $ H^0 $, i.e., $ ({H^0})^g = H^0 $. Conversely, if $ g \in N_G(H^0) $, then 
\[
{\cal H}^g = \left( \bigcap_{b \in N_G(H^0)} H^b \right)^g = \bigcap_{b \in N_G(H^0)} H^{bg} = {\cal H}.
\]
\end{proof}

\begin{rmk}\label{rmk-NGH=NGH0}
Clearly, $ {\cal H} $ is also an $ M $-definable dfg component of $ G $. By replacing $ H $ with $ {\cal H} $, we may assume that $ N_G(H) = N_G(H^0) $. 
\end{rmk}

Recall from \cite{PYZ-$p$-adic-groups} that 
\begin{dfn}\label{def-DAC}
    A \emph{definably amenable component} of  $G$ is a maximal definably amenable subgroup containing a dfg component of $G$.
\end{dfn}

The following fact was established in Theorem 9.3 of \cite{PYZ-$p$-adic-groups} for $ p $-adically closed fields $ k $; the proof, however, extends verbatim to the case where $ k $ is a real closed field.

\begin{fact}\label{fact-max-DA-subgroup}
Let $ H $ a dfg component of $ G $, and $ B = N_G(H^0) $. Then:
\begin{enumerate}
    \item [(i)] $B$ is a definably amenable component and $ B = N_G(B^{00}) $. 
    \item [(ii)] Any two definably amenable components of $G$ are conjugate.
    \item [(iii)] $ G $ is definably amenable if and only if $ G = B $.
\end{enumerate}
\end{fact}


\begin{rmk}
While definably amenable components are maximal among definably amenable subgroups containing a dfg component of $ G $, they need not be the largest with respect to dimension. For instance: Let $ \mathbb{M} \models \pcf $, $ G = \mathrm{SL}_2(\mathbb{M}) $, and let $ H $ denote the subgroup of upper triangular matrices in $ G $. Then $ H $ is a dfg component of $ G $, and $ B = N_G(H) = H $. Consider $ O = \mathrm{SL}_2(\mathcal{O}(\mathbb{M})) $, where $ \mathcal{O}(\mathbb{M}) $ is the valuation ring of $ \mathbb{M} $. Since $ O $ is definably compact, it is definably amenable. However, $ \dim(O) = 3 $, which is greater than $ \dim(H) = 2 $.
\end{rmk}

The following fact is taken from Lemma 8.11 of \cite{PYZ-$p$-adic-groups}, where $\M\models \pcf$. The proof, however, also applies to the case where $ \M$ is an $o$-minimal expansion of $\mathrm{RCF} $.

\begin{fact}\label{fact-dfg-components-over k}
    Suppose that $ E \leq D \leq G $ are definable groups such that $ G/E $, $ D/E $, and $ G/D $ are definable. If $ D/E $ and $ G/D $ are definably compact, then $ G/E $ is definably compact.
\end{fact}

\begin{lemma}\label{lemma-G/B=A/V}
  Let $ G $ and $ X $ be definable groups, and let $ f: G \to X $ be a definable surjective morphism with $ \ker(f) $ definably amenable. If $ Y $ is a definably amenable component of $ X $, then $ f^{-1}(Y) $ is a definably amenable component of $ G $.
\end{lemma}

\begin{proof}
Let $ Z \leq Y $ be a dfg component of $ X $. Clearly, $ G/f^{-1}(Z) \cong X/Z $ is definable and, by Fact \ref{fact-main-conj-RCF} and Fact \ref{fact-main-conj-pCF}, definably compact. Let $ H $ be a dfg component of $ f^{-1}(Z) $; then $ f^{-1}(Z)/H $ is definably compact. By Fact \ref{fact-dfg-components-over k}, $ G/H $ is also definably compact, whence $ H $ is a dfg component of $ G $ by Facts \ref{fact-main-conj-RCF} and \ref{fact-main-conj-pCF}.

Since $ f^{-1}(Y) $ fits into the definable short exact sequence 
\[ 1 \to \ker(f) \to f^{-1}(Y) \to Y \to 1 \]
with $ \ker(f) $ and $ Y $ both definably amenable, it follows that $ f^{-1}(Y) $ is definably amenable. It remains to show that $ f^{-1}(Y) $ is maximal among all definably amenable subgroups of $ G $ containing $ H $. Suppose for a contradiction that there exists a definably amenable subgroup $ B_1 \gvertneqq f^{-1}(Y) $ that contains $ H $. Then $ f(B_1) \gvertneqq f(f^{-1}(Y)) = Y $ is definably amenable, which contradicts the fact that $ Y $ is a definably amenable component of $ X $.
\end{proof}

\begin{lemma}\label{lemma-semisimple-dfg-comp-decom}
 Let $A \subseteq \mathbb{M}^n$ be an $M$-definable semisimple group. Then there exists a dfg component $W$ of $A$ such that:
\begin{enumerate}
    \item[(i)] $W$ decomposes as $W = U(\mathbb{M}) \rtimes W_t$, where $U$ is a unipotent algebraic group, $T = \clz{W_t}$ is an $M$-split torus, both over $M$, and $\clz{W} = U \rtimes T$.
    \item[(ii)] If $A$ has finite index in $\clz{A}(\mathbb{M})$, then $\clz{W}(\mathbb{M})$ is a dfg component of $\clz{A}(\mathbb{M})$.
\end{enumerate}
\end{lemma}

\begin{proof}
Let $X$ be an $M$-definable dfg component of $A$. By Facts \ref{fact-dfg-group-Qp}, \ref{fact-dfg-group-R}, and \ref{fact-dfg-alg-group}, the Zariski closure $\clz{X}$ is an $M$-split solvable linear algebraic group over $M$, hence triangularizable over $M$. We may therefore write $\clz{X} = U \rtimes T$, where $U$ is unipotent and $T$ is an $M$-split torus, both defined over $M$.

Let $X_u = X \cap U$ and $X_t = X \cap T$. Since $\dim(X) = \operatorname{Zdim}(\clz{X})$, $X$ is open in $\clz{X}(\mathbb{M})$. By the subspace topology, $X_u$ and $X_t$ are open in $U(\mathbb{M})$ and $T(\mathbb{M})$, respectively.
By Lemma \ref{lemma-open-is-finite-index}, $X_u$ has finite index in $U(\mathbb{M})$ and $X_t$ has finite index in $T(\mathbb{M})$. Since $U(\mathbb{M})$ and $T(\mathbb{M})$ have dfg, so do $X_u$ and $X_t$. Thus both $X_u$ and $X_t$ are open dfg subgroups of $U(\mathbb{M})$ and $T(\mathbb{M})$, respectively.

Note that $U(\mathbb{M})$ is definably connected as pointed out in Remark \ref{rmk-unip-split}, it follows that $X_u = U(\mathbb{M})$.

Since $T$ normalizes $U$, the subgroup $X_t$ normalizes $U(\mathbb{M})$. The semidirect product $X_u \rtimes X_t$ is therefore a well-defined definable dfg subgroup of $A$. Since $X_u \rtimes X_t$ is an open dfg subgroup of $X$, it is a dfg component of $A$. Set $W_u = X_u$, $W_t = X_t$, and $W = W_u \rtimes W_t$.

By construction, $W_u \subseteq U$, $W_t \subseteq T$, and $W \subseteq U \rtimes T$; moreover, $\dim(W_u) = \operatorname{Zdim}(U)$, $\dim(W_t) = \operatorname{Zdim}(T)$, and $\dim(W) = \dim(U \rtimes T)$. Hence $\clz{W_u} = U$, $\clz{W_t} = T$, and $\clz{W} = U \rtimes T$. This establishes part (i).

We now prove part (ii). Suppose that $A$ has finite index in $\clz{A}(\mathbb{M})$. First, $A/W$ is definably compact. Since $W \subseteq \clz{W}(\mathbb{M})$, the canonical projection $\pi \colon A/W \to A/(\clz{W}(\mathbb{M}))$ is surjective, and thus $A/(\clz{W}(\mathbb{M}))$ is definably compact.
Since $(\clz{A}(\mathbb{M}))/A$ is finite, hence definably compact, it follows from Fact \ref{fact-dfg-components-over k} that $(\clz{A}(\mathbb{M}))/(\clz{W}(\mathbb{M}))$ is definably compact.
By Fact \ref{fact-dfg-group-R}, it remains to verify that $\clz{W}(\mathbb{M})$ has dfg. Since $W$ has dfg and $W$ is open in $\clz{W}(\mathbb{M})$ (as they have the same dimension), it follows from Lemma \ref{lemma-open-is-finite-index} that $\clz{W}(\mathbb{M})$ has dfg. This completes the proof.
\end{proof}

From now on, we will fix the following notations:

\begin{notations}\label{notations-A-V-W-B}
Throughout this section, $G$ denotes a non-definably amenable group definable over $M$. Recall from Corollary \ref{lem-amenable-semi-simple-decom} that $G$ fits into a short exact sequence of $M$-definable groups
\begin{equation}\label{equ-D-G-A-notation}
1 \to D \to G \stackrel{\pi_A}{\to} A \to 1,
\end{equation}
where $D$ is definably amenable, $A$ is linear, its Zariski closure $\clz{A}$ is semisimple, and $A $ has finite index in $\clz{A}(\M)$.

\begin{itemize}
    \item Let $W$ be an $M$-definable dfg component of $A$, decomposes as in Lemma \ref{lemma-semisimple-dfg-comp-decom}: $W = W_u \rtimes W_t$;
    \item $\clz{W_u}$ is unipotent, $T=\clz{W_t}$ is an $M$-split torus, and $W_u=\clz{W_u}(\M)$
    \item $\clz{W}= \clz{W_u} \rtimes \clz{W_t}$ and $\clz{W}(\M)$ is  a dfg component of $\clz{A}(\M)$. 
    \item Let $V = N_A(W_u)$ and $B = \pi_A^{-1}(V)$. 
\end{itemize}
\end{notations}
We prove later that $V$ and $B$ are the definably amenable components of $A$ and $G$, respectively (see Claims \ref{claim-V=NAW} and \ref{claim-V=NAV}). With these notations fixed, we have the following claims:

\begin{claim}\label{calim-clzV}
 $\clz{N_A(W_u)}  = N_{\clz{A}}(\clz{W_u}) $ and $\clz{N_A(W)}=N_{\clz{A}}(\clz{W})$
\end{claim}
\begin{proof}
We only show that $\clz{N_A(W_u)}  = N_{\clz{A}}(\clz{W_u})$, and the latter one is similar.
Note that $\clz{W_u}^g=\clz{W_u^g}$ for any $g\in A$. If $g\in V=N_A(W_u)$, then $\clz{W_u}^g=\clz{W_u^g}=\clz{W_u}$, so $N_A(W_u)\leq N_{\clz{A}}(\clz{W_u})$, and hence $\clz{N_A(W_u)}\leq N_{\clz{A}}(\clz{W_u})$. Since $V=N_A(W_u)$, this gives $\clz{V}\leq N_{\clz{A}}(\clz{W_u})$.

Conversely, $(N_{\clz{A}}(\clz{W_u}))(\mathbb{M})=\{g\in \clz{A}(\mathbb{M})\mid \clz{W_u}^g=\clz{W_u}\}$.
For any $g\in \clz{A}(\mathbb{M})$, $\clz{W_u}^g=\clz{W_u}$ holds if and only if $(\clz{W_u}(\mathbb{M}))^g=\clz{W_u}(\mathbb{M})$, i.e., $W_u^g=W_u$. Therefore $(N_{\clz{A}}(\clz{W_u}))(\mathbb{M})\leq V$. Taking Zariski closures,
\[
N_{\clz{A}}(\clz{W_u})=\clz{\big((N_{\clz{A}}(\clz{W_u}))(\mathbb{M})\big)}\leq \clz{V},
\]
as required.
\end{proof}

\begin{claim}\label{fact-Wug=Wu}
    For any $ g \in \clz{A} $, if $ \clz{W_u}^g \leq \clz{W} $, then $ \clz{W_u}^g = \clz{W_u} $. As a consequence, for any $ g \in A $, if  $ W_u^g \leq W $, then $ W_u^g = W_u $.
\end{claim}
\begin{proof}
Since $ \clz{W} = \clz{W_u} \rtimes \clz{W_t} $, the subgroup $ \clz{W_u} $ coincides with the set of unipotent elements in $ \clz{W} $. Moreover, conjugates of unipotent elements remain unipotent. Consequently, if $g\in \clz{A}$ such that $ \clz{W_u}^g \leq \clz{W} $, then $ \clz{W_u}^g $ must consist of unipotent elements in $ \clz{W} $, which implies $ \clz{W_u}^g \leq \clz{W_u} $.

Since $ \clz{W_u} $ is connected (as an algebraic group) and $ \clz{W_u}^g $ has the same Zariski dimension as $ \clz{W_u} $, we conclude that $ \clz{W_u}^g = \clz{W_u} $, as required.

Now let $ g \in A $ and suppose $ W_u^g \leq W $. This implies $ \clz{W_u}^g = \clz{W_u^g} \leq \clz{W} $. By the preceding argument, $ \clz{W_u}^g = \clz{W_u} $, which in turn implies 
\[
(\clz{W_u}(\M))^g = (\clz{W_u}^g)(\M) = \clz{W_u}(\M).
\]
Since $ W_u = \clz{W_u}(\M) $, $W_u^g=(\clz{W_u}(\M))^g$. Thus $ W_u^g = W_u $, as required.
\end{proof}

\begin{fact}\label{fact-red-alg-gp}
Let $X$ be an algebraic group over $M$.  
\begin{enumerate} 
    \item Let $Y$ be an algebraic subgroup of $X$ over $M$. If $X(\M)/Y(\M)$ is definably compact, then $Y$ contains a maximal $M$-split solvable algebraic subgroup. $\mathrm{[see \ \cite{PYZ-$p$-adic-groups}\  Corollary\ 8.20 \ and}$ $ \mathrm{  \cite{AG-and-NT}\ Theorem\ 3.1]}$
    \item If $X$ is reductive, then the intersection of two Borel subgroups of $X$ contains a maximal torus. $\mathrm{[see\  \cite{T.A. Springer}\  Corollary \ 8.3.10]}$
    \item Let $P$ and $Q$ be two conjugate parabolic subgroups of $X$, whose intersection contains a Borel subgroup. Then $P=Q$. $\mathrm{[see \cite{T.A. Springer} \ Corollary\  6.4.11]}$
    \item Let $P$ be a parabolic subgroup of $X$. Then $N_X(P)=P$. $\mathrm{[see\  \cite{T.A. Springer} \ Corollary\  6.4.10]}$
    \item Let $P$ be a parabolic subgroup of $X$. Then $N_X(R_u(P))=P$. $\mathrm{[see\ \cite{T.A. Springer}\ §6.4]}$
\end{enumerate}
\end{fact}

\begin{rmk}
Regarding Part (1) of Fact \ref{fact-red-alg-gp}, Theorem 3.1 of \cite{AG-and-NT} proves the case when $ k = \mathbb{Q}_p $ or $ k = \mathbb{R} $. Corollary 8.20 of \cite{PYZ-$p$-adic-groups}, which is only proven for $ p $-adically closed fields. However, its proof also applies to the case of real closed fields.
\end{rmk}

\begin{claim}\label{claim-V=NAW}
$V = N_A(W^0)$. Consequently, $V$ is a definably amenable component of $A$.
\end{claim}
\begin{proof}
\noindent\textit{Step 1: $N_A(W^0) \leq V$:}
Note that $W^0 = W_u^0 \rtimes W_t^0 = W_u \rtimes W_t^0$, since $W_u$ is definably connected. Let $g \in N_A(W^0)$. Conjugation by $g$ maps $W_u$ into $W$, so $W_u^g \leq W$. By Claim \ref{fact-Wug=Wu}, we have $W_u^g = W_u$, which implies $g$ normalizes $W_u$. Since $V = N_A(W_u)$, it follows that $g \in V$. Thus, $N_A(W^0) \leq V$.
\vspace{0.2cm}

\noindent\textit{Step 2: $V \leq N_A(W^0)$:}
Since $\clz{W}(\M)$ is a dfg component of $\clz{A}(\M)$, the quotient
$\clz{A}/\clz{W}$ is definably compact. By Fact \ref{fact-red-alg-gp} (1),
$\clz{W}$ contains a maximal $M$-split solvable algebraic subgroup; as
$\clz{W}$ is itself $M$-split solvable, it is maximal. Write
$\clz{W}=T\ltimes U$, where $T=\clz{W_t}$ is a maximal $M$-split torus and
$U=\clz{W_u}$ is unipotent.

Then $P=Z_{\clz{A}}(T)U$ is a minimal $M$-parabolic of $\clz{A}$ with
$R_u(P)=U$. By Fact \ref{fact-red-alg-gp} (2), $N_{\clz{A}}(U)=P$.

Let $g\in V=N_A(W_u)$. Since $W_u=U(\M)$ and $U(\M)$ is Zariski dense in
$U$, $g$ normalizes $U$, so $g\in N_{\clz{A}}(U)=P$.

Now $P=Z_{\clz{A}}(T)\ltimes U$. The subgroup $Z_{\clz{A}}(T)$ centralizes
$T$, hence normalizes $W_t^0\le T(\M)$; it also normalizes $U=R_u(P)$,
hence normalizes $W_u=U(\M)$. Thus $Z_{\clz{A}}(T)$ normalizes
$W^0=W_u\rtimes W_t^0$. Meanwhile $U(\M)=W_u\le W^0$, so every element of
$U$ normalizes $W^0$. Therefore every element of $P=Z_{\clz{A}}(T)U$
normalizes $W^0$, and in particular $g\in N_A(W^0)$. Hence
$V\le N_A(W^0)$.

Combining Steps 1 and 2, we obtain $V = N_A(W^0)$, as required.
\end{proof}

\begin{claim}\label{claim-V=NAV} 
  $B$ is a definably amenable component of $G$. 
\end{claim}
\begin{proof}
By  Claim \ref{claim-V=NAW}, $V$ is a definably component of $A$. By Lemma \ref{lemma-G/B=A/V} (2), $\pi_A^{-1}(V)$ is a definably amenable component of $G$.
\end{proof}

\begin{claim}\label{claim-D-sq-B}
$D \le B$.
\end{claim}
\begin{proof}
Since $B$ is a definably amenable component of $G$, it contains a dfg component $H$ of $G$.
As $D$ is a normal definably amenable subgroup of $G$, the product $BD$ is also a definably amenable subgroup of $G$.
Since $BD$ contains $H$, the maximality of $B$ as a definably amenable component implies $BD = B$, and thus $D \le B$.
\end{proof}



\subsection{The compact quotient space and the dfg component invariant types}\label{H(k0)-invariant types}
In this subsection, $M_0$ always denotes the standard model of $\mathrm{Th}(M)$, i.e.\ $M_0 = \mathbb{Q}_p$ or $M_0=(\mathbb{R},<,\ldots)$, and we assume that $M$ is an elementary extension of $M_0$ and is a small model. 

 Let $ G $ be a group definable over $ M_0 $, and let $ \mu_G $ be the intersection of all $ M_0 $-definable open neighborhoods of the identity in $ G $. Then $ \mu_G $ is a type-definable subgroup of $ G $. Let $ V_G = G(M_0)\mu_G $; this is a subgroup of $ G $, and there is a standard part map $ \st : V_G \to G(M_0) $ that maps each $ g \in V_G $ to the unique $ g_0 \in G(M_0) $ such that $ g \in g_0\mu_G $. It is easy to see that $ V_G \leq N_G(\mu_G) $. Let $ E $ be a $ M_0 $-definable subgroup of $ G $ such that $ X = G/E $ is definable and definably compact (equivalently, $ X(M_0) $ is compact). We regard $ X $ as the space of left $ G $-cosets of $ E $; for any $ g \in G $, $ g/E $ denotes the left coset $ gE $, so $ X = \{g/E \mid g \in G\} $.

Let $O$ be a $M_0$-definable open subset of $G$ such that $\cl{O}$ is definably compact. Equivalently, $O(M_0)$ is a definable open   subset of $G(M_0)$ and $\cl{O(M_0)}$ is compact. Let $\pi_X: G\to X$ be the nature projection. Then $\pi_X$ is an open map and thus $\pi(gO)=\{gh/E|\ h\in O\}$ is open in $X$ for any $g\in G$. Since $\{\pi(gO(M_0))|\ g\in G(M_0)\}$ is an open cover of $X(M_0)$, we see that there are $g_1,....,g_n\in G(M_0)$ such that $X=\bigcup_{i=1}^n\pi_X(g_iO)$.  Let $U=\cl{\bigcup_{i=1}^n g_iO} $, then $U$ is $M_0$-definable, definably compact,  and   $G=UE$. thus the restriction $\pi_X: U \to X $ of the projection $\pi_X$ to $U $ is surjective.

Without lose of generality, we may assume that $\mu_G\sq U$. Since   $U(M_0)$ is a compact definable subset of $G(M_0)$, we see that $U\sq U(M_0)\mu_G \sq V_G$. For any $ x \in X $, $ \st(x) $ denotes the unique element $ y \in X(M_0) $ such that $ x \in U $ for all $ M_0 $-definable neighborhoods of $ y $. It is easy to see that $ \st(g/E) = \st(g)/E $ for all $ g \in V_G $.

\begin{lemma}\label{epsilon-equ}
Let notations be as in the above. Suppose that $M\succ M_0$. Then for any $g,h\in G$ with $\tp{(g/E)/M}=\tp{(h/E)/M}$, there exists $\epsilon_1\in \mu_G$ such that $\epsilon_1 g /E=h /E$. Moreover, if $g \in V_G$, then there exists $\epsilon_2\in \mu_G$ such that $g\epsilon_2 /E=h /E$.
\end{lemma}

\begin{proof}
Take $g_0,h_0\in U$ satisfying $g_0/E = g/E$ and $h_0/E = h/E$. It suffices to show $\st(g_0)/E = \st(h_0)/E$. For each $M_0$-definable open neighborhood $Z$ of $\st(g_0)$, we have that $Z/E$ is an $M_0$-definable open neighborhood of $g_0/E$. Since $\tp{(g_0/E)/M} = \tp{(h_0/E)/M}$, it follows that $h_0/E \in Z/E$. Thus, $h_0/E$ lies in every $M_0$-definable open neighborhood of $\st(g_0)/E$, which implies $\st(h_0)/E = \st(g_0)/E$ as required.

If $g \in V_G$, then $\mu_G g = g\mu_G$ as $\mu_G$ is normal in $V_G$, so there exists $\epsilon_2\in \mu_G$ such that $g\epsilon_2 = \epsilon_1 g$. This yields $g\epsilon_2/E = \epsilon_1 g/E = h/E$.
\end{proof}

\begin{lemma}\label{kid}
Let notations be as in the above. Let $p \in S_G(M_0)$ be $E^0(M_0)$-invariant. Then for any $u\in U $ and $h\in E $ such that $uh\models p$, we have ${E^0(M_0)}^{u}\subseteq \mu_G  E$.
\end{lemma}
\begin{proof}
    Let $g\in E^0(M_0)$, then  $\tp{guh/M_0}=\tp{uh/M_0}$, which implies that $\tp{(gu/E)/M_0}=\tp{(u/E)/M_0}$. By Lemma \ref{epsilon-equ}, there is $\epsilon\in \mu_G$ such that $gu\epsilon /E= u/E$. We see that $g^u\epsilon\in E $. Note that $g^u\in V_G$, it follows that $g^u\epsilon=\epsilon'g^u$ for some $\epsilon'\in \mu_G$, and therefore $g^u\in \mu_G  E $. So ${E^0(M_0)}^{u}\subseteq \mu_G  E $ as required.
\end{proof}

\begin{lemma}\label{klem-1}
Let the notations be as in Notations \ref{notations-A-V-W-B}, except that all objects are definable over $M_0$. Let $p\in S_A(M)$ be $W^0(M_0)$‑invariant. Suppose $g\models p$; then $g=\epsilon v$ for some $\epsilon\in \mu_A$ and $v\in V$, where $\epsilon,v\in \dcl{M_0, g}$.
\end{lemma}
\begin{proof}
First, note that $W^0(M_0)=W_u^0(M_0)\rtimes W_t^0(M_0)$. Since $W_u$ is definably connected, $W_u^0=W_u$.

As discussed at the beginning of Section \ref{H(k0)-invariant types}, since $W$ is a dfg component of $A$, there exists an $M_0$-definable, definably compact subset $Z$ such that $A=ZW$. We may therefore write $g=zw$ with $z\in Z$ and $w\in W$.  Let $z_0$ be the standard part of $z$.

Since $\tp{gw/M_0}=\tp{zw/M_0}$ is $W^0(M_0)$‑invariant, Lemma \ref{kid} yields $W^0(M_0)^z\subseteq \mu_A W$. As $Z\subseteq V_A=\mu_A A(M_0)$, we have $W^0(M_0)^z\le V_A$. We deduce
\[
W_u(M_0)^{z_0}\le W^0(M_0)^{z_0}
=\st\bigl(W^0(M_0)^z\bigr)
\subseteq \st\bigl(\mu_A W \cap V_A\bigr)
= W(M_0).
\]
Hence $W_u^{z_0}\le W$. By Claim \ref{fact-Wug=Wu}, $W_u^{z_0}=W_u$, so $z_0\in V=N_A(W_u)$. Therefore $g=zw=\epsilon z_0 w$ for some $\epsilon\in \mu_A$, and $z_0 w\in V$. Set $v:=z_0 w$. Since $\Th(M_0)$ has definable Skolem functions, we may take $\epsilon, v\in \dcl{M_0, g}$.
\end{proof}

\begin{fact}[see \cite{JY-2} Lemma 2.1]\label{fact-pi(00)=pi()00}
    Let $h:X\to Y$ be a surjective homomorphism of definable groups. Then $f(X^{00})=Y^{00}$.
\end{fact}

\begin{lemma}\label{klem}
Let the notations be as in Notations \ref{notations-A-V-W-B}, except that all objects are definable over $M_0$, and let $H$ be a normal dfg component of $B$. Let $M \succ M_0$ and $p \in S_G(M)$ be $H^0(M_0)$-invariant. Suppose that $g \models p$; then $g = \epsilon b$ for some $\epsilon \in \mu_G$ and $b \in B$ with $\epsilon, b \in \dcl{M_0,g}$.
\end{lemma}

\begin{proof}
We first claim that $\pi_A(H)^0 = W^0$: Since $H$ is a normal dfg component of $B$,  we have that $B/H$ is a definably compact group, and so is $V/\pi_A(H)$. Thus, $\pi_A(H)$ is a normal dfg component of $V$. By Facts \ref{fact-main-conj-RCF} and \ref{fact-main-conj-pCF}, we have $\pi_A(H)^0 = W^0$.

Now, it is easy to see that $q = \pi_A(p)$ is $\pi_A(H^0(M_0))$-invariant. By Fact \ref{fact-pi(00)=pi()00},  $(\pi_A(H))^0 = \pi_A(H^0)$, so $q$ is $(\pi_A(H))^0(M_0)$-invariant, namely $W^0(M_0)$-invariant. By Lemma \ref{klem-1}, $\pi_A(g) = \epsilon' v$ for some $\epsilon' \in \mu_A$ and $v \in V$. Now, $\pi_A(gB) = \pi_A(g)V = \epsilon' V \in \mu_A V$. Since $\pi_A: G \to A$ is open, we have $\pi_A(\mu_G) = \mu_A$. Thus, $\pi_A^{-1}(\mu_A V) = \mu_G B D = \mu_G B$ by Claim \ref{claim-D-sq-B}. Therefore, $g/B \in \mu_G/B$, so there exists $\epsilon \in \mu_G$ such that $g \in \epsilon b$. Since $\Th(M_0)$ has definable Skolem functions, we may assume that $\epsilon, b \in \dcl{M_0,g}$.
\end{proof}

\begin{cor}\label{exchange}
Let notations be as in Notations \ref{notations-A-V-W-B} and $p\in S_G({M^{\ext}})$ be $B^{00}(M)$-invariant. Let $g \models p $. Then there exist  $\epsilon \in \mu_G $ and  $b \in B $  such that $g = \epsilon b\in \mu_GB$.
\end{cor}
\begin{proof}
Since $\tp{g/{M^{\ext}}}$ is $B^{00}(M)$-invariant, thus $\tp{g/{M^{\ext}}}$ is $H^0(M_0)$-invariant. By Lemma \ref{klem}, there exist  $\epsilon \in \mu_G $ and  $b \in B $  such that $g = \epsilon b\in \mu_GB$.
\end{proof}

\subsection{The Ellis group of definably amenable components}\label{sec-Ellis Group-B}
Let notations be as in Notations \ref{notations-A-V-W-B}, $H$ a normal dfg component of $B$ such that $ N_G(H) = N_G(H^0) $ (see Remark \ref{rmk-NGH=NGH0}). Let $X=G/H$, then $X$ is definably compact as $H$ is a dfg component of $G$.  Let $ \pi: G \to X $ be the natural projection (i.e., $ \pi $ is the map $ \pi_X: G \to X $ mentioned in the second paragraph of Section \ref{H(k0)-invariant types}), and let $ f $ be a definable section of $ \pi$. Then   $ B $ admits a definable exact sequence:
\begin{equation}\label{equ-dss-H-B-C}
   1 \to H \to B \to C \to 1. 
\end{equation}
where $ C= B/H $ is a definably compact group.
  Note that for any $ c \in C $, take $ b \in B $ such that $ \pi(b) = c $. Then $ f(c)/H = b/H $, which is equivalent to $ f(c)^{-1}b \in H $; this implies $ f(c) \in B $. Thus $ f(C) \subseteq B $, so the restriction of $ f $ to $ C $ is also a definable section of the group homomorphism $ \pi: B \to C $. Define $ \eta: C \times C \to B $ by
\[
\eta(c_1, c_2) = f(c_1  c_2)^{-1}  f(c_1)  f(c_2).
\]
Since
\[
\pi(\eta(c_1,c_2))=\pi(f(c_1  c_2)^{-1}  f(c_1)  f(c_2)) = (c_1  c_2)^{-1}c_1c_2 =\id_C
\]
for all $ c_1, c_2 \in C $, we conclude $ \eta(C \times C) \subseteq H $.

 In this subsection, we compute the Ellis group of $B$ over $M$, or equivalently, construct an ideal group $\eb \subseteq S_B(M^\text{ext})$. We fix the following notation:

\begin{notations}\label{notations-minimal-flows}
Let $\ic$ be the unique minimal subflow of the $C(M)$-flow $S_C(M^\text{ext})$ (see Fact \ref{fact-fsg-generic-space}), and let $\ih$ be a minimal subflow of the $H(M)$-flow $S_H(M^\text{ext})$.
\end{notations}

 Note that by Fact \ref{Mini-flow-of-B}, $ \ih $ is an Ellis group of $ S_H(M^{\text{ext}}) $. Clearly, $ f $ extends naturally to a map from $ S_C(M^{\text{ext}}) $ to $ S_B(M^{\text{ext}}) $; we still denote this extension by $ f $. Then we have the following fact:

\begin{fact}[\cite{YZ-APAL}, Lemma 3.9]
$ f(\ic) * \ih = \{ f(q) * p \mid q \in \ic, \, p \in \ih \} $ is a minimal subflow of $ S_B(M^{\text{ext}}) $.
\end{fact}

\begin{rmk}
Note that if $\cal M $ is a minimal subflow of $ S_G(M^{\text{ext}}) $ and $ p \in \mathcal{M} $, then $ p * \mathcal{M} $ is the Ellis group containing $ p $.
\end{rmk}

\begin{lemma}\label{lemma-I}
Let $v\in\ic$,  then  $f(v*\ic)*\ih=f(v)*f(\ic)*\ih$
\end{lemma}
\begin{proof}
Let $q\in \ic$ and $p\in \ih$. Let $c\models v$, $c'\models q|({M^{\ext}}, c)$, and $h\models p|({M^{\ext}},c,c')$. then 
    \[
    f(v*q)*p=\tp{f(c c') h/{M^{\ext}}}=\tp{f(c) f(c')\cdot(\eta(c,c')^{-1} h)/{M^{\ext}}} 
    \]
    Note that by Fact \ref{B(Mext)-has-dfg}, $\tp{ (\eta(c,c')^{-1} h)/{M^{\ext}},c,c'}$ is an heir of some $p'\in \ih$, thus we have that 
     \[
    \begin{split}
    & \tp{f(c) f(c') (\eta(c,c')^{-1} h)/{M^{\ext}}}\\
    &= \tp{f(c) f(c')/{M^{\ext}}}*p'\\
    &=\big(\tp{f(c)/{M^{\ext}}}*\tp{ f(c')/{M^{\ext}}}\big)*p'\in f(v)*f(\ic)*\ih,
    \end{split}
    \]
    as required. Conversely, take $p''\in \ih$ such that $\tp{ (\eta(c,c')  h)/{M^{\ext}},c,c'}$ is an heir of  $p''$, then we have
\[
    \begin{split}
    f(v)*f(q)*p&=\tp{f(c   c') \eta(c,c')  h/{M^{\ext}}}\\
    &=    \tp{f(c  c')/{M^{\ext}}}*p''\\
    &=   f\big(\tp{c  c'/{M^{\ext}}}\big)*p''\\
    &=  f\big(\tp{c/{M^{\ext}}}* \tp{ c'/{M^{\ext}}}\big)*p'',
    \end{split}
    \]
    which is in $f(v *\ic)*\ih$.
    \end{proof}

\begin{notations}
Let $\mathfrak{u}\in \ic$ be an idemponent  and $\ec=\mathfrak{u}*\ic$ the ideal group of $S_C({M^{\ext}})$ containing $\mathfrak{u}$.     
\end{notations}

\begin{lemma}\label{lemma-II} 
There is $p\in \ih$ such that $f(\mathfrak{\mathfrak{u}})*p$ is an idemponent. 
\end{lemma}

\begin{proof}


Let $N$ be an $|M|^+$-saturated elementary extension of $M^\ext$ such that $H(N)$ intersects every coset of $H^0$ in $H$. 
Let $c_1,c_2\in C(N)$ and $b^*\in B$ such that $(c_1,c_2)\models \mathfrak{u}\otimes \mathfrak{u}$ and $b^*\models f(\mathfrak{u})| N$. Then there is $\delta\in H(N)$ such that $\delta^{b^*}/H^0=\eta(c_1,c_2)^{-1}/H^0$. Let $p\in \ih$ such that $p/H^0=\delta/H^0$.

We claim that $f(\mathfrak{u})*p$ is an idemponent: let $a_1\models \mathfrak{u}$, $h_1\models p|({M^{\ext}},a_1)$, $a_2\models \mathfrak{u}|(N, a_1,h_1)$, and $h_2\models p|(N,a_1,a_2,h_1)$, then 
\[
\begin{split}
  f(\mathfrak{u})*p*f(\mathfrak{u})*p &= \tp{f(a_1) h_1 f(a_2) h_2/{M^{\ext}}}\\
  &= \tp{f(a_1) f(a_2) h_1^{f(a_2)}  h_2/{M^{\ext}}}\\
  &=\tp{f(a_1 a_2) \eta(a_1,a_2)  h_1^{f(a_2)}  h_2/{M^{\ext}}}.
\end{split}
\]
Since $H^0$ is normal in $B$, we see that  $h/H^0\mapsto h^{f(a_2)}/H^0$ is an automorphism of $H/H^0$. It follows that $h_1^{f(a_2)}/H^0= \delta^{f(a_2)}/H^0$ as $h_1 /H^0= p/H^0=\delta /H^0$.
Since $\tp{{b^*}/N}=\tp{f(a_2)/N}$, we have that $\tp{{\delta^{b^*}}/N}=\tp{\delta^{f(a_2)}/N}$, which implies that $\delta^{b^*}/H^0=\delta^{f(a_2)}/H^0$. Therefore, 
\[
h_1^{f(a_2)}/H^0= \delta^{f(a_2)}/H^0=\delta^{b^*}/H^0=\eta(c_1,c_2)^{-1}/H^0.
\]
Since \[
\tp{c_1,c_2/{M^{\ext}}}=\tp{a_1,a_2/{M^{\ext}}}=\mathfrak{u}\otimes \mathfrak{u},\]
we have that $\eta(c_1,c_2)^{-1}/H^0=\eta(a_1,a_2)^{-1}/H^0$. So we conclude that 
\[
\eta(a_1,a_2) h_1^{f(a_2)}=\eta(c_1,c_2) h_1^{f(a_2)}\in H^0.
\]
It follows that 
\[
\begin{split}
f(\mathfrak{u})*p*f(\mathfrak{u})*p &=\tp{f(a_1 a_2) \eta(a_1,a_2)  h_1^{f(a_2)}  h_2/{M^{\ext}}}\\
&=\tp{f(a_1  a_2)/{M^{\ext}}}*\tp{\eta(a_1,a_2)  h_1^{f(a_2)}  h_2/{M^{\ext}}}\\
&=f\big(\tp{a_1/{M^{\ext}}}*\tp{ a_2/{M^{\ext}}})\big)*p\\
&=f(\mathfrak{u}*\mathfrak{u})*p=f(\mathfrak{u})*p.
\end{split}
\]
\end{proof}

\begin{notations}
    Let $\eb=f(\ec) * \ih$.
\end{notations}

\begin{pro}\label{pro-III}
$\eb $ is an ideal group of $S_B({M^{\ext}})$.  
\end{pro}
\begin{proof}
By Lemma \ref{lemma-II}, there is $p\in \ih$ such that $f(\mathfrak{u})*p$ is an idemponent. It suffices to show that $\eb=f(\mathfrak{u})*p*f(\ic)*\ih$.

Let $v \in \ic$ and $q \in \ih$. Take $(a_1,h_1,a_2,h_2)\models \mathfrak{u}\otimes p\otimes v\otimes q$. Then:
\begin{equation}\label{equ-computation}
\begin{split}
f(\mathfrak{u}) * p * f(v) * q &= \tp{f(a_1) h_1 f(a_2) h_2 / {M^{\ext}}}\\
&= \tp{f(a_1)  f(a_2)  h_1^{f(a_2)}  h_2 / {M^{\ext}}}\\
&= \tp{f(a_1  a_2)  \eta(a_1, a_2)  h_1^{f(a_2)}  h_2 / {M^{\ext}}}\\
&= \tp{f(a_1  a_2) / {M^{\ext}}} * \tp{\eta(a_1, a_2)  h_1^{f(a_2)}  h_2 / {M^{\ext}}} 
\end{split}
\end{equation}

Now, we have:
\[
\tp{f(a_1  a_2) / {M^{\ext}}} = f\left( \tp{a_1 / {M^{\ext}}} * \tp{a_2 / {M^{\ext}}} \right) = f(\mathfrak{u} * v) \in f(\ec)
\]
and
\[
\tp{\eta(a_1, a_2) h_1^{f(a_2)}  h_2 / {M^{\ext}}} = \tp{\eta(a_1, a_2)  h_1^{f(a_2)} / {M^{\ext}}} * q \in \ih.
\]
We thus conclude that $f(\mathfrak{u}) * p * f(v) * q \in f(\ec) * \ih$, which further implies 
\[
f(\mathfrak{u}) * p * f(\ic) * \ih \subseteq f(\ec) * \ih.
\]

Next, we show that $f(\ec) * \ih \subseteq f(\mathfrak{u}) * p * f(\ic) * \ih$. Recall from Lemma \ref{lemma-I} that $f(\ec) * \ih = f(\mathfrak{u}) * f(\ic) * \ih$. Take any $v \in \ic$. 
Let $N$ be an $|M|^+$-saturated elementary extension of $M^\ext$ such that $H(N)$ intersects every coset of $H^0$ in $H$. Let $(a_1, h, a_2) \models (\mathfrak{u} \otimes p \otimes v) \mid N$. Choose $\delta \in H(N)$ such that $h^{f(a_2)} / H^0 = \delta / H^0$. For $q \in \ih$, let $s = \tp{\delta^{-1} / {M^{\ext}}} * q$, and take $h^* \models s \mid (N, \delta, a_1, a_2, h)$. Then:
\[
\begin{split}
f(\mathfrak{u}) * p * f(v) * s &= \tp{f(a_1)  f(a_2)  h^{f(a_2)}  h^* / {M^{\ext}}}\\
&= \tp{f(a_1) f(a_2) / {M^{\ext}}} * \tp{h^{f(a_2)} h^* / {M^{\ext}}}. 
\end{split}
\]

Since 
\[
h^{f(a_2)}  h^* / H^0 =(h^{f(a_2)}    / H^0)(h^*/H^0)=(\delta/H^0)(\delta^{-1}/H^0)(q/H^0)=  q / H^0,
\]
it follows that $\tp{h^{f(a_2)}  h^* / {M^{\ext}}} = q$ as $\tp{h^{f(a_2)}  h^* /{M^{\ext}}}$ and  $ q  $ are contained in $\ih$. Thus:
\[
\tp{f(a_1)  f(a_2) / {M^{\ext}}} * \tp{h^{f(a_2)}  h^* / {M^{\ext}}} = \tp{f(a_1)  f(a_2) / {M^{\ext}}} * q = f(\mathfrak{u}) * f(v) * q.
\]
This implies that for any $v\in \ic$ and $q\in \ih$, there is $s\in \ih$ such that 
\[
f(\mathfrak{u}) * f(v) * q=f(\mathfrak{u}) * p * f(v) * s.
\]
We conclude that 
\[
f(\ec) * \ih = f(\mathfrak{u}) * f(\ic) * \ih \subseteq f(\mathfrak{u}) * p * f(\ic) * \ih.
\]
This completes the proof.
\end{proof}

Directly following Proposition \ref{pro-III} and the computation in Equation (\ref{equ-computation}), we have:
\begin{cor}\label{cor-v*q*eb=eb}
For any $v\in \ec$ and $q\in S_B(M^\ext)$, we have $f(v)*q*\ih\sq \eb$ and  $f(v)*q*\eb\sq \eb$. 
\end{cor}

\subsection{The group of infinitesimals}

Let $\mu$ denote the intersection of all $M_0$-definable open neighborhoods of $0$ in $\mathbb{M}$. Specifically, if $\mathbb{M}$ is an o-minimal expansion of a real closed field, then
\[
\mu = \{x \in \mathbb{M} \mid \forall r \in \mathbb{R}^{>0}\ (-r < x < r)\};
\]
if $\mathbb{M}$ is a $p$-adically closed field, then
\[
\mu = \{x \in \mathbb{M} \mid \forall n \in \mathbb{Z}\ (v(x) > n)\}.
\]
Let $V = M_0 + \mu$, and let $\operatorname{st}: V \to M_0$ be the standard part map.
Given a set $C \subseteq M_0^n$, we define the \emph{hull} of $C$ to be the set $C^\mathrm{hull} := C + \mu^n$, where $\mu^n$ denotes the $n$-fold Cartesian product of $\mu$.

\begin{fact}[\cite{Dries-standard}]
  Let $\M$ be an $o$-minimal expansion of $\mathrm{RCF}$ and $ D \subseteq \M^n $ be  definable.
  Then  
  \begin{enumerate}
     \item $ D \cap \r^n $ is definable in $ \r $ and $ \dim (D \cap \r^n) \leq \dim(D) $;  
     \item $ \st(D\cap V^n) $ is definable in $ \r $ and $ \dim (\st(D\cap V^n)) \leq \dim(D) $. 
  \end{enumerate}   
\end{fact}

\begin{fact}[\cite{Dries-standard}]\label{fact-standard-part-map}
Let $\mathbb{M}$ be an o-minimal expansion of $\mathrm{RCF}$, and $f : \mathbb{M}^m \to \mathbb{M}$ an $\mathbb{M}$-definable function. Then there exists a finite partition $\mathcal{P}$ of $\mathbb{M}^m$ into definable sets, each of which is either open in $\mathbb{M}^m$ or has empty interior. On each open set $C \in \mathcal{P}$, one of the following holds:
\begin{enumerate}
\item $f(x) \notin V$ for all $x \in C^\mathrm{hull}$;
\item there exists a continuous function $g : C \to \mathbb{M}$, definable in $\mathbb{M}$, such that $f(x) \in V$ and $\operatorname{st}(f(x)) = g(\operatorname{st}(x))$ for all $x \in C^\mathrm{hull}$.
\end{enumerate}
\end{fact}

\begin{lemma}\label{lemma-st-part-cap-dim}
Let $ U \subseteq \mathbb{R}^n $ be an open cell such that $ U(\M) \subseteq V^n $. Let $ Y, Z $ form a definable partition of $ U(\mathbb{M}) $. Then 
\[
\dim\left( \st(Y) \cap \st(Z) \right) < n.
\]
\end{lemma}

\begin{proof}
We proceed by induction on $ n $. For $ n = 1 $, the statement is straightforward to verify. Suppose now $ n > 1 $. Assume $ U $ has the form
\[
U = \left\{ (a, b) \in \mathbb{R}^n \mid a \in E,\ h_1(a) < b < h_2(a) \right\},
\]
where $ E \subseteq \mathbb{R}^{n-1} $ is an open cell. Without loss of generality, let $ Y $ be an open cell of the form
\[
Y = \left\{ (a, b) \in \mathbb{M}^n \mid a \in D,\ f_1(a) < b < f_2(a) \right\},
\]
where $ D \subseteq \M^{n-1} $ is an open cell and $ f_1, f_2 $ are $ \M $-definable functions. Then $ Z $ decomposes as:
\[
\begin{split}
Z = &\left\{ (a, b) \in U(\M) \mid a \in D,\ b \leq f_1(a) \right\} \\
&\cup \left\{ (a, b) \in U(\M) \mid a \in D,\ f_2(a) \leq b \right\} \\
&\cup \left\{ (a, b) \in U(\M) \mid a \in E(\M) \setminus D \right\}.
\end{split}
\]

By Fact \ref{fact-standard-part-map}, there exists a finite partition
\[
\mathcal{P} = \{ C_1, \ldots, C_l, C_{l+1}, \ldots, C_m \}
\]
of $ \mathbb{R}^{n-1} $ into definable sets, where $ C_1, \ldots, C_l $ are open in $ \mathbb{R}^{n-1} $ and $ C_{l+1}, \ldots, C_m $ have empty interior. For each $ 1 \leq j \leq l $ and $ i = 1, 2 $, there is a continuous $ \mathbb{R} $-definable function $ g_{i, C_j}: C_j \to \mathbb{R} $ such that $ \st(f_i(x)) = g_{i, C_j}(\st(x)) $ for all $ x \in C_j^\text{hull}   $.  

We observe that:
\[
\st(Y) \subseteq \bigcup_{j=1}^l \left\{ (a, b) \in U \mid a \in \st(D),\ g_{1, C_j}(a) \leq b \leq g_{2, C_j}(a) \right\} 
\cup \bigcup_{j=l+1}^m \left( \st(C_{j}) \times \mathbb{R} \right),
\]
and
\[
\begin{split}
\st(Z) \subseteq &\bigcup_{j=1}^l \left\{ (a, b) \in U \mid a \in \st(D),\ b \leq g_{1, C_j}(a) \right\} \\
&\cup \bigcup_{j=1}^l \left\{ (a, b) \in U \mid a \in \st(D),\ g_{2, C_j}(a) \leq b \right\} \\
&\cup \left\{ (a, b) \in U \mid a \in \st(E(\M) \setminus D) \right\} \cup \bigcup_{j=l+1}^m \left( \st(C_{j}) \times \mathbb{R} \right).   
\end{split}
\]

Taking their intersection, we derive:
\[
\begin{split}
\st(Y) \cap \st(Z) &\subseteq \bigcup_{j=1}^l \left\{ (a, b) \in U \mid a \in \st(D),\ b = g_{1, C_j}(a) \right\} \\
&\cup \bigcup_{j=1}^l \left\{ (a, b) \in U \mid a \in \st(D),\ b = g_{2, C_j}(a) \right\} \\
&\cup \left\{ (a, b) \in U \mid a \in \st(E(\M) \setminus D) \cap \st(D) \right\} 
\cup \bigcup_{j'=l+1}^m \left( \st(C_{j'}) \times \mathbb{R} \right).
\end{split}
\]

By the induction hypothesis, $ \dim\left( \st(E(\M) \setminus D) \cap \st(D) \right) < n-1 $, so 
\[
\dim\left( \left\{ (a, b) \in U \mid a \in \st(E(\M) \setminus D) \cap \st(D) \right\} \right) < n.
\]
All other sets in the union are either graphs of continuous functions (hence of dimension $ n-1 $) or subsets of $ \st(C_{j}) \times \mathbb{R} $ with $ \dim(\st(C_{j})) < n-1 $ (so their product has dimension $ < n $). Thus, $ \dim\left( \st(Y) \cap \st(Z) \right) < n $, as required.
\end{proof}

Note that for the $ M_0 $-definable group $ G $, there exists a $ M_0 $-definable open neighborhood $ O $ of $ \id_G $ and a $M_0$-definable homeomorphism $ \chi: O  \to \M^n $ (for some $ n \in \mathbb{N} $) with $\chi(\id_G)=(0,\cdots,0)$. It is straightforward to verify that  $ \chi(\mu_G) = \mu^n $  and  $ \chi(O(M_0)\mu_G) = V^n $. Moreover, the following diagram commutes: 
\[\xymatrix{O(M_0)\mu_G \ar[r]^-{\chi} \ar[d]_-{\st} & V^n \ar[d]^-{\st } \\O(M_0) \ar[r]^-{\chi } & M_0^n}.\]
Thus, for any $ L_{\mathbb{M}} $-formula $ \tau(x) $ satisfying $ \tau(\mathbb{M}) \subseteq O(M_0)\mu_G \sq V_G$, we may further identify $ \tau(\mathbb{M}) $ with its image $ \chi(\tau(\mathbb{M})) \subseteq \mathbb{M}^n $. Keeping these notations, we have:
\begin{pro}\label{pro:muG-fsg}
The type-definable group $\mu_G$ has the  fsg property. Specifically, there exists a global type $r \in S_G(\M)$ such that:
\begin{enumerate}
    \item $r \vdash \mu_G$ (i.e., every formula in $\mu_G(x)$ is contained in $r$);
    \item $r$ is finitely satisfiable in $M_0$;
    \item For all $g \in \mu_G$, the left and right translates satisfy $g  r = r  g = r$.
\end{enumerate}
\end{pro}

\begin{proof}
If $ M_0 = \Q $, then by Fact \ref{fact-open-compact-subgp}, $ G $ has an open $ \Q $-definable subgroup $ X $ such that $ X(\Q) $ is compact. We observe that $ X $ has fsg and $ \mu_G =\mu_X=  X^{00}$. Let $ r \vdash X^{00} $ be a global fsg type of $ X $; then $ r $ clearly satisfies the required properties.

If $ M_0 = \mathbb{R} $,  let $ r_0 \in S_G(\r) $ satisfy $ r_0 \vdash \mu_G $ and $ \dim(r_0) = \dim(G) $. Since $ r_0 $ is definable, it has a unique global coheir $ r \in S_G(\M) $. It is straightforward to verify that for any $ L_\M $-formula $ \psi(x) $, $ \psi(x) \in r $ if and only if $ \psi(\r) \in r_0 $. We establish the following claim:

\begin{Claim}
Let $ \varphi(x) $ be an $ L_\M $-formula. Then $ \varphi \in r $ if and only if $ \st(\varphi(\M) \cap V_G) \in r_0 $.
\end{Claim}

\begin{proof}
Let $\tau$ be an $L_{\M}$-formula such that $\mu_G\sq \tau(\M)\sq O(M_0)\mu_G$. Without loss of generality, assume that $ \varphi(\M)\sq\tau(\M)$ is an open cell.  Since $ \varphi(\r) \subseteq \st(\varphi(\M)) $, if $ \varphi \in r $, then $ \st(\varphi(\M)) \in r_0 $.

Conversely, suppose $ \st(\varphi(\M)) \in r_0 $. If $ \neg\varphi \in r $, then $ \st(\neg\varphi(\M) \cap \tau(\M) ) \in r_0 $, which implies
\[
\st(\varphi(\M)) \cap \st(\neg\varphi(\M) \cap \tau(\M)) \in r_0.
\]
By Lemma \ref{lemma-st-part-cap-dim},
\[
\dim\left( \st(\varphi(\M)) \cap \st(\neg\varphi(\M) \cap \tau(\M)) \right) < n,
\]
contradicting $ \dim(r_0) = n $.
\end{proof}

We now show that $ \epsilon \cdot r = r = r \cdot \epsilon $ for any $ \epsilon \in \mu_G $. Let $ \varphi(x) \in r $; it suffices to show $ \epsilon\varphi(x) \in r $ and $ \varphi(x)\epsilon \in r $. Assume that $ \varphi(\M) \sq O(M_0)\mu_G $. By the above claim, $ \varphi(x) \in r $ iff $ \st(\varphi(\M)) \in r_0 $. Since
\[
\st(\epsilon\varphi(\M)) = \st(\varphi(\M)\epsilon) = \st(\varphi(\M)),
\]
we have $ \varphi(x) \in r \iff \epsilon\varphi(x) \in r \iff \varphi(x)\epsilon \in r $. This completes the proof.
\end{proof}

\subsection{The Ellis group of $G$}\label{sec-Ellis Group-G}
In this section, we assume that $G$ is an $M_0$-definable group that is not definably amenable, and $B$ is a definably amenable component of $G$ that is definable over $M_0$. Consistent with the assumptions at the beginning of Section \ref{sec-Ellis Group-B}, $B$ admits a definable exact sequence as specified in (\ref{equ-dss-H-B-C}); $f: C \to B$ is a definable section of $\pi: B \to C$; $\ec$, $\ih$, and $\eb = f(\ec) * \ih$ (see Proposition \ref{pro-III}) are the ideal groups of $S_C(M^\ext)$, $S_H(M^\ext)$, and $S_B(M^\ext)$, respectively. Let $ \mu_G $, $ \mu_B $, $ \mu_H $, and $ \mu_C $ denote the intersections of all $ M_0 $-definable open neighborhoods of the identity elements of $ G $, $ B $, $ H $, and $ C $, respectively.

The proof technique employed herein follows the same logical framework as the computation of Ellis groups presented in \cite{BY-APAL}, while we have streamlined the proof steps for greater conciseness.

In Lemma 2.3.3 of  \cite{BY-APAL}, they proved that the following Fact \ref{kconjg} holds in $\mathbb{Q}_p$, and the same method can show that this Fact also holds in an $o$-minimal expansion of $\mathbb{R}$.
	\begin{fact}\label{kconjg}
		Let $g \in G$ and $\epsilon \in \mu_G$ such that $\tpo{g}{{M^{\ext}}, \epsilon}$ is finitely satisfiable in $M_0$, then $\epsilon^g \in \mu_G$. Particularly, if $N\succ M$ and $g \in G$ such that $\tpo{g}{N}$ is finitely satisfiable in $M_0$, then $(\mu_G(N))^g \leq \mu_G$.
	\end{fact}

\begin{lemma}\label{ctp}
Let  $\epsilon_0, \epsilon \in \mu_G$ and  $b_0, b \in B$ such that $\epsilon_0b_0 = \epsilon b$, then $b_0, b$ are in the same coset of $B^{00}$.
\end{lemma}
\begin{proof}
 We have	$bb_0^{-1} = \epsilon^{-1}\epsilon_0 \in \mu_G \cap B=\mu_B$. Since $\mu_B\leq B^{00}$, we see that  $b_0$ and $ b$ are in the same coset of $B^{00}$.
\end{proof}



  By Fact \ref{fact-definable-amenable} (iv), $\tau: p\mapsto p/B^{00}$ is an isomorphism from $\eb$ to $B/B^{00}$.  For each $b \in B$,  consider the map 
 \[
 l_b: \eb \to \eb,\ p \mapsto p_b*p,
 \]
 where $p_b$ is the unique element of $ \eb$ such that $p_b\vdash b /B^{00}$. Then $l_b$ is a bijection from $\eb$ to itself, and $l_{b_1}=l_{b_2}$ iff $b_1/B^{00}=b_2/B^{00}$. In addition, for any $b_1, b_2 \in B$, we have $l_{b_1} \circ l_{b_2} = l_{(b_1  b_2)}$. 

\begin{notations}\label{notations-pB}
  We now fix $ p_B $ as the identity/idempotent of $ \eb $, i.e., $ p_B \vdash B^{00} $.
\end{notations}

 It is easy to verify that $ l_{b_1}(p_B) * l_{b_2}(p_B) = l_{(b_1  b_2)}(p_B) $ for  $ b_1, b_2 \in B $.




\begin{lemma}\label{kb}
Let  $ q \in S_G(M^\ext) $ and $ b_0 \in B $. If $ p_B * q \vdash \mu_G b_0 B^{00} $, then $ l_b(p_B) * q \vdash \mu_G b b_0 B^{00} $ for any $ b \in B $.
\end{lemma}

\begin{proof}
First, note that $ l_b(p_B) * q = l_b(p_B) * p_B * q $.

Let $ N \succ M^\ext $ be $ |M|^+ $-saturated. Take  $ \epsilon \in \mu_G(N) $ and $ b^* \in B^{00}(N) $ such that  $  \epsilon b_0 b^*\models  p_B * q$.

Let $ v \in \ec $ and $h\in H(N)$ such that $f(v)*\tp{h/M^\ext}\vdash b/B^{00}$. We now choose $ h_0 \in H(N) $ such that $ h_0 / H^0 = h / H^0 $ and $ \tp{h_0 / M^\ext, \epsilon, b_0, b^*} $ is finitely satisfiable in $ M_0 $. Since $H^0\leq B^{00}$,  we see that $f(v)*\tp{h_0/M^\ext}\vdash b/B^{00}$.

By Corollary \ref{cor-v*q*eb=eb}, $f(v)*\tp{h_0/M^\ext}*p_B\in \eb$. Since $f(v)*\tp{h_0/M^\ext}*p_B\vdash b/B^{00}$, we have  that $ f(v)*\tp{h_0/M^\ext}*p_B=l_b(p_B)$. Let $ c \in C $ realize the unique generic extension of $ v $ over $ N $, we have:
\[
\begin{split}
\tp{f(c) h_0 \epsilon b_0 b^* / M^\ext} 
&= f(u) * \tp{h_0 / M^\ext} * \tp{\epsilon b_0 b^* / M^\ext}\\
&=f(u) * \tp{h_0 / M^\ext} *p_B * q \\
&= l_b(p_B)*p_B*q
.    
\end{split}
\]
Next, rewrite the element $ f(c) h_0 \epsilon b_0 b^* $ using conjugation:
\[
f(c) h_0 \epsilon b_0 b^* = \epsilon^{(f(c) h_0)^{-1}} f(c) h_0 b_0 b^*.
\]
By  Fact \ref{kconjg},  $ \epsilon^{(f(c) h_0)^{-1}} \in \mu_G(N) $. Thus:
\[
f(c) h_0 \epsilon b_0 b^* \in \mu_G f(c) h_0 b_0 B^{00} = \mu_G b b_0 B^{00}.
\]
This completes the proof.
\end{proof}

\begin{notations}\label{notations-rG}
Let $ r_G \in S_G(M^\ext) $ be an fsg type on $ \mu_G $ given by Proposition \ref{pro:muG-fsg}.    
\end{notations}
 Note that $ r_G $ is an idempotent. 
As $ p_B * r_G $ is $ B^{00}  $-invariant, it follows from Corollary \ref{exchange} that there is some $  b \in B $ such that $ p_B * r_G \vdash \mu_G b B^{00} $. 
\begin{notations}\label{notations-qB}
   We fix an element $\mathfrak{b} \in B$ such that $p_B * r_G \vdash \mu_G \mathfrak{b} B^{00}$, and $q_B = l_{\mathfrak{b}^{-1}}(p_B)$.
\end{notations}

By Lemma \ref{kb}, we immediately obtain the following corollary:

\begin{cor}\label{left-action-on-idempotent}
For each $ b \in B $, $ l_b(q_B) * r_G \vdash \mu_G b B^{00} $. In particular, $ q_B * r_G \vdash \mu_G B^{00} $.
\end{cor}
	
We now prove that $r_G * q_B$ is an idempotent in a minimal subflow of $S_G(M^\ext)$. 


\begin{lemma}\label{qpqp}
Suppose $p \in \eb$ and $b \in B$. Then  
\[
r_G * l_b(q_B) * r_G * p = r_G * l_{b}(p).
\]
In particular, $r_G*p=r_G*q_B*r_G*p=r_G*p*r_G*q_B$.
\end{lemma}
\begin{proof}
By Corollary \ref{left-action-on-idempotent}, $l_b(q_B) * r_G \vdash \mu_G b B^{00}$. Then  
\[
r_G * l_b(q_B) * r_G * p = r_G * \tp{\epsilon b'/{M^{\ext}}} * p
\]
for some $\epsilon \in \mu_G$ and $b' \in B$ such that $b'/B^{00}=b/B^{00}$. It is easy to see that 
\[
r_G * \tp{\epsilon b'/{M^{\ext}}} = r_G * \tp{b'/{M^{\ext}}}
\]
as $r_G$ is finitely satisfiable generic. Thus, 
\begin{equation}\label{equ-1-1}
r_G * l_b(q_B) * r_G * p = r_G * \tp{b'/{M^{\ext}}} * p.
\end{equation}
Take any $v \in \ec$; then $f(v)$ is realized by some $g \in f(C) \sq V_G$. It is easy to see that  
\[
r_G = r_G * (\st(g)^{-1} \cdot f(v)) = (r_G \cdot \st(g)^{-1}) * f(v)
\]
as $\st(g)^{-1} \cdot f(v) \vdash \mu_G$. So 
\begin{equation}\label{equ-1-2}
r_G * \tp{b'/{M^{\ext}}} * p = (r_G \cdot \st(g)^{-1}) * (f(v) * \tp{b'/{M^{\ext}}} * p).
\end{equation}
By Corollary \ref{cor-v*q*eb=eb}, 
$f(v) * \tp{b'/{M^{\ext}}} * p \in \eb$, and thus 
\[
\tp{\st(g)^{-1}/M^\ext} * f(v) * \tp{b'/{M^{\ext}}} * p = \st(g)^{-1} \cdot (f(v) * \tp{b'/{M^{\ext}}} * p) \in \eb.
\]
Since $(\st(g)^{-1} \cdot f(v)) \vdash \mu_G\cap B$ and $\mu_G\cap B=\mu_B \leq B^{00}$, we have that 
\[
\big(\st(g)^{-1} \cdot (f(v) * \tp{b'/{M^{\ext}}} * p)\big)/B^{00} = \big(\tp{b'/{M^{\ext}}} * p\big)/B^{00} = l_b(p)/B^{00},
\]
which implies that $\st(g)^{-1} \cdot (f(v) * \tp{b'/{M^{\ext}}} * p) = l_b(p)$. Combining this with equations (\ref{equ-1-1}) and (\ref{equ-1-2}), we have that $r_G * l_b(q_B) * r_G * p = r_G * l_b(p)$, as required.
\end{proof}

By Lemma \ref{qpqp}, we see immediately that
\begin{cor}
For any $ p, q \in \eb $, we have $ r_G * p * r_G * q = r_G * p * q $. In particular,  $ r_G * q_B * r_G * q_B = r_G * q_B $, i.e. , $ r_G * q_B $ is an idempotent.
\end{cor}
	
	We now show that the subflow $ S_G(M^{\ext}) * r_G * q_B $ generated by $ r_G * q_B $ is minimal. 

\begin{rmk}
Recall that $\pi: G\to G/H$ is the canonical quotient projection, and $f$ is a definable section of $\pi$. Let $U$ be the $M_0$-definable subset of $G$ given in the second paragraph of Section \ref{H(k0)-invariant types}; explicitly, $U$ satisfies $\mu_G \subseteq U$, $G = UH$, and $U(M_0)$ is compact.  Since the restriction $\pi\upharpoonright_U: U\to G/H$ is surjective, we may assume that the image of the definable section $f$ lies inside $U$. In particular, $f(C)\subseteq U\cap B$. 
\end{rmk}
	\begin{lemma}\label{kj}
		Suppose that $p\in\ih$. Then 
		\[
		S_G({M^{\ext}})*p\sq S_{U}({M^{\ext}}) *\ih.
		\]
	\end{lemma}
	\begin{proof}
	Let $s \in S_G({M^{\ext}})$. 	Let $N $ be an $|M|^+$-saturated extension of ${M^{\ext}}$. Take $u_0\in U(N )$,  $h_0\in H(N )$, and $h\in H$ such that $u_0h_0 \models s$  and $h \models p  | N$. Then:
		$$s * p  = \tpo{u_0h_0h}{{M^{\ext}}}.$$

		By Fact \ref{B(Mext)-has-dfg},   there is $q \in \ih$ such that $\tpo{h_0h}{N }=h_0(p  \upharpoonright N)=q \upharpoonright N $, which implies that $(u_0,h_0h)\models \tp{u_0/{M^{\ext}}}\otimes \tp{h_0h/{M^{\ext}}}$. Therefore we have:
        \[
        s * p =\tpo{u_0 h_0h}{{M^{\ext}}}=\tpo{u_0  }{{M^{\ext}}}*\tpo{  h_0h}{{M^{\ext}}}\in S_{U}({M^{\ext}})*\ih.
        \]
        This completes the proof.
	\end{proof}

\begin{pro}\label{AP}
$S_G(M^{\ext}) * r_G * q_B$ is a minimal subflow, i.e. , for any $ s \in S_G(M^{\ext}) $, we have:
\[
r_G * q_B \in S_G(M^{\ext}) * s * r_G * q_B = \cl{G(M) \cdot (s * r_G * q_B)}.
\]
\end{pro}
\begin{proof}
By Lemma \ref{kj}, we may assume that $s * r_G * q_B = q * p$ where $p = \ih$ and $q \in S_{U}({M^{\ext}})$. Let $N $ be an $|M|^+$ saturated extension of ${M^{\ext}}$. Let $ u \in  U(N )$ such that $u \models q$ and $q'=\tp{u^{-1}/{M^{\ext}}}$, then  $q'*q\vdash \mu_G$, and hence, $r_G=r_G*(q'*q)$. It follows that    
\[
r_G*q'*s * r_G * q_B =r_G*q'* q * p=r_G*p.
\]
 Let  $v\in \ec$ , then  $f(v)\in S_{U\cap B}(M^\ext)$ . Let  $v_0\in U(M_0)\cap B(M_0)$  be the standard part of  $f(v)$. Then
\[
f(v)*\tp{v_0^{-1}/M^\ext}=f(v)\cdot v_0^{-1}\vdash\mu_G.
\]
It follows that $r_G*f(v)*\tp{v_0^{-1}/M^\ext}=r_G$, and hence $r_G*p=r_G*f(v)*\tp{v_0^{-1}/M^\ext}*p$.
By Corollary \ref{cor-v*q*eb=eb}, we have $p'=f(v)*\tp{v_0^{-1}/M^\ext}*p\in \eb$. So $r_G*p=r_G*p'$ with $p'\in \eb$. Let  $b\in B$  be such that  $l_B(p')=q_B$ . By Lemma \ref{qpqp},
\[
r_G*l_b(q_B)*r_G*p'=r_G*l_b(p')=r_G*q_B.
\]
Combining the above, we obtain
\[
r_G*l_b(q_B)*r_G*q'*s * r_G * q_B
=r_G*l_b(q_B)*(r_G*p)
=r_G*l_b(q_B)*(r_G*p')
=r_G*q_B.
\]
Thus  $r_G*q_B\in S_G({M^{\ext}})*s * r_G * q_B$ , as required.
\end{proof}

\begin{notations}\label{notation-ig-rG-q_B}
 Let  $\ig=S_G(M^\ext) * r_G * q_B$ be the minimal subflow generated by $r_G * q_B$.       
\end{notations}
We  now compute the ideal group $\eg=r_G * q_B*\ig$ of $S_G({M^{\ext}})$, i.e. the ideal group generated by $r_G * q_B$.
	
\begin{lemma}   
$r_G * q_B * \ig = r_G * \eb$
  
\end{lemma}

\begin{proof}
Let $p \in \eb$. By Lemma \ref{qpqp},
\[r_G *p=   (r_G*q_B)*(r_G*p)*(r_G * q_B)\in r_G * q_B * \ig.\]
Thus, we have  $r_G * \eb \subseteq r_G * q_B * \ig$.
		
Now we show that $r_G * q_B * \ig \subseteq r_G * \eb$. 
By lemma \ref{kj}, 
\[
\ig=S_G({M^{\ext}})*(r_G*q_B)\sq S_{U}({M^{\ext}})*\ih.
\]
So it suffices to show 
\[
r_G * q_B * S_{U}({M^{\ext}}) * \ih \subseteq r_G * \eb.
\]
Let $q \in S_{U}({M^{\ext}})$ and $p_H \in \ih$. By Lemma \ref{klem}, $q_B*q\vdash \mu_GB$. Let $N \succ {M^{\ext}}$ be $|M|^+$-saturated. Suppose that $q_B*q=\tp{\epsilon b/{M^{\ext}}}$ where $\epsilon\in \mu_G(N)$ and $b\in B(N)$.  Suppose that $b=f(c)h$ with $c\in C(N)$ and $h\in H(N)$. Let $p_H'=\tp{h/{M^{\ext}}}*p_H$, then $p_H'\in \ih$ and
\[
r_G * q_B *q*p_H=r_G*\tp{\epsilon f(c)h/{M^{\ext}}}*p_H=r_G*\tp{ f(c) /{M^{\ext}}}* p_H'.
\]
Take any $v\in \ec$ and let $g\models f(v)$, then $r_G=r_G*f(v)*\tp{\st(g^{-1})/M^\ext}$.
So 
\[
r_G*\tp{ f(c) /{M^{\ext}}}* p_H'= r_G*f(v)*\tp{\st(g^{-1})/M^\ext}* p_H'
\]
By Corollary \ref{cor-v*q*eb=eb}, $f(v)*\tp{\st(g^{-1})/M^\ext}* p_H'\in \eb$. Therefore, $r_G * q_B *q*p_H\in r_G*\eb$ as required.
\end{proof}

\begin{theorem}\label{theorem-rG*eb}
    $r_G * \eb$ is isomorphic to $\eb$ as abstract groups.  
\end{theorem}

\begin{proof}
Recall that $p_B \vdash B^{00}$ is the identity element of $\eb$, so $l_b(p_B) * l_{b'}(p_B) = l_{bb'}(p_B)$ for any $b, b' \in B$. 
Define a map $\xi: \eb \to r_G * \eb$ by $\xi\big(l_b(p_B)\big) = r_G * l_b(q_B)$ for $b \in B$. 
By Lemma \ref{qpqp}, for any $b, b' \in B$,
\[
\big(r_G * l_b(q_B)\big) * \big(r_G * l_{b'}(q_B)\big) = r_G * l_{bb'}(q_B),
\]
which shows that $\xi$ is well-defined and is a group isomorphism.
\end{proof}

\begin{cor}
    Let $X$ be a group definable over $M_0$. Then the Ellis group of $X$ over $M$ is isomorphic to the Ellis group of its definably amenable component over $M$.
\end{cor}
\begin{proof}
    If $X$ is not definably amenable, the statement follows directly from Theorem \ref{theorem-rG*eb}. If $X$ is definably amenable, then $X$ coincides with its own definably amenable component, so the conclusion holds trivially.
\end{proof}

\subsection{Newelski's conjecture}
    We conclude this paper by showing that when $G$ is definable over a $p$-adically closed field, then $G$ is definably amenable iff the Newelski's conjecture holds.
    
\begin{lemma}\label{lemma-conjecture-equ}
Let $G$ be a group definable over $\Q$ and $B$ a definably amenable component of $G$ over $\Q$. Then Newelski's conjecture holds for $G$ if and only if $B^{00} = B \cap G^{00}$.
\end{lemma}
\begin{proof}
By Theorem \ref{theorem-rG*eb}, there is an ideal group $\eb$ of $S_B(M^\ext)$ such that $\eg=r_G*\eb$ is an ideal group in $S_G(M^\ext)$. Let $\tau: \eb \to r_G * \eb$ be the map defined by $p \mapsto r_G * p$. Then \[
\tau(l_b(p_B))=r_G*l_b(p_B)=r_G*l_b(l_\mathfrak{b}(q_B))=\xi(l_{b}(l_{\mathfrak{b}}(p_B)))\]
for any $b\in B$, so $\tau$ is a bijection (though not an isomorphism). Consider the the canonical epimorphisms arising in Newelski's conjecture:
\[
\mathfrak{n}_G: r_G * \eb \to G/G^{00},\  r_G * p \mapsto (r_G * p)/G^{00}.
\]
Since $r_G \vdash \mu_G$ and $\mu_G \leq G^{00}$, it follows that $(r_G * p)/G^{00}=(r_G/G^{00})(p/G^{00}) = p/G^{00}$ for any $p \in \eb$. Thus, $\mathfrak{n}_G(\tau(p)) = p/G^{00}$ for any $p \in \eb$, so the following diagram commutes:
\[
\xymatrix{ \eb \ar[r]^-{\tau} \ar[d]_-{\mathfrak{n}_B} & r_G*\eb \ar[d]^-{\mathfrak{n}_G } \\ B/B^{00} \ar[r]^-{\mathfrak{i} } & G/G^{00}},
\]
where $\mathfrak{n}_B: \eb \to B/B^{00}$ is also the canonical epimorphisms arising in Newelski's conjecture and $\mathfrak{i}: B/B^{00} \to G/G^{00}$ is the map defined by $b/B^{00} \mapsto b/G^{00}$. Note that $\mathfrak{n}_B$ is an isomorphism, as $B$ is definably amenable (hence Newelski's conjecture holds for $B$). Since both $\mathfrak{n}_B$ and $\tau$ are bijective, we see that $\mathfrak{n}_G$  is injective if and only if $\mathfrak{i}$ is injective. Now,  Newelski's conjecture holds for $G$ iff $\mathfrak{n}_G$ is injective, and $\mathfrak{i}$ is injective if and only if $B^{00} = B \cap G^{00}$. This completes the proof.
\end{proof}



\begin{lemma}\label{lemma-new-for-X-for-G}
 Let $X$ be a group definable over $M$, $G$ a group definable over $M_0$, and $f \colon X \twoheadrightarrow G$ an $M$-definable surjective homomorphism. If Newelski's conjecture fails for $G$ over $M$, then it fails for $X$ over $M$.
\end{lemma}
\begin{proof}
 Let $B$ be a definably amenable component of $G$ over $M_0$ and $E = f^{-1}(B)$, then $E$ is a definably amenable component of $X$. Let $\eg = r_G * \eb \subseteq S_G(M^\text{ext})$ be an ideal group given in Theorem \ref{theorem-rG*eb}. It is straightforward to see that there exist an ideal group $\ex \subseteq S_X(M^\text{ext})$ such that $f(\ex) = \eg$ and an ideal group $\ee \subseteq S_E(M^\text{ext})$ such that $f(\ee) = \eb$. Let $u_X \in \ex$ be the identity element, then $\ex = u_X * S_X(M^\text{ext}) * u_X$.

Since Newelski's conjecture fails for $G$, $B^{00}$ is a proper subgroup of $G^{00} \cap B$. Let $b \in (G^{00} \cap B) \setminus B^{00}$. Since $f^{-1}(B) = E$ and $f(X^{00}) = G^{00}$, there exists $a \in (X^{00} \cap E) \setminus E^{00}$ such that $f(a) = b$.

 Then $u_X * p * u_X \vdash X^{00}$: indeed $u_X \vdash X^{00}$, and $p \vdash X^{00}$ since $a \in X^{00}$ and $E^{00} \le X^{00}$. Let $r_G * q_B$ be the identity of $\eg$, then  
    \[
    f(u_X * p * u_X) = (r_G * q_B) * f(p) * (r_G * q_B) = r_G * (q_B * f(p)).
    \]

Since $p/E^{00} = a/E^{00}$, we have that $f(p)/B^{00} = f(a)/B^{00} = b/B^{00}$. So $f(p) \not\vdash B^{00}$, and hence $r_G * (q_B * f(p)) \neq r_G * q_B$.   Thus $u_X * p * u_X$ is a non-identity element of $\ex$ that maps to the trivial coset $X^{00}$ under the canonical map, so Newelski's conjecture fails for $X$ over $M$.
\end{proof}

\begin{lemma}\label{lemma-new-for-X-for-Y}
Let $X$ and $Y$ be groups definable over $M$. If $X$ is a finite index subgroup of $Y$, then Newelski's conjecture holds for $X$ if and only if it holds for $Y$.
\end{lemma}
\begin{proof}
First, it is easy to verify that $X^{00}=Y^{00}$.

Suppose $y_1X, \dots, y_nX$ are distinct cosets of $X$ in $Y$.
We claim the following:
\begin{Claim}
$v \in S_X(M^\ext)$ is almost periodic (namely, in a minimal flow) in $S_X(M^\ext)$ if and only if it is almost periodic in $S_Y(M^\ext)$.
\end{Claim}
\begin{proof}
Let $v \in S_X(M^\ext)$ be almost periodic.
We show that $S_Y(M^\ext)*v$ is a minimal subflow of $S_Y(M^\ext)$. Take any $s \in S_Y(M^\ext)$; it suffices to verify $v \in S_Y(M^\ext)*s*v$, by Fact \ref{fact-minimal-ideal}. Suppose $s \vdash y_iX$; then $y_i^{-1}s \in S_X(M^\ext)$. Since $S_X(M^\ext)*v$ is minimal, there exists $r \in S_X(M^\ext)$ such that $v = r*y_i^{-1}s*v$, so $v \in S_Y(M^\ext)*s*v$, as required.

Conversely, suppose $v \in S_X(M^\ext)$ is almost periodic in $S_Y(M^\ext)$. Take any $q \in S_X(M^\ext)$; it suffices to show $v \in S_X(M^\ext)*q*v$. Since $v$ is almost periodic in $S_Y(M^\ext)$, there exists $r \in S_Y(M^\ext)$ such that $v = r*q*v$. As $q, v \vdash X$, it follows that $r \vdash X$, so $r \in S_X(M^\ext)$ and hence $v \in S_X(M^\ext)*q*v$, as required.
\end{proof}
Let $v \in S_X(M^\ext)$ be an almost periodic idempotent. Then by the claim, $\ex = v*S_X(M^\ext)*v$ and $\ey = v*S_Y(M^\ext)*v$ are ideal groups of $S_X(M^\ext)$ and $S_Y(M^\ext)$, respectively. Let $\mathfrak{n}_X: \ex\rightarrow X/X^{00}$ and $\mathfrak{n}_Y: \ey\rightarrow Y/Y^{00}$ be the canonical epimorphisms arising in Newelski's conjecture for $X$ and $Y$, respectively.
Since $X^{00}=Y^{00}$, the restriction of $\mathfrak{n}_Y$ to $\ex$ is precisely $\mathfrak{n}_X$. 

\begin{Claim}
$\ker(\mathfrak{n}_X)=\ker(\mathfrak{n}_Y)$.
\end{Claim}
\begin{proof}
It suffices to show that $\ker(\mathfrak{n}_Y)\leq \ex$.  Let $e\in \ker(\mathfrak{n}_Y)$. Since $Y^{00}=X^{00}\leq X$, the type $e$ concentrates on $X$. Moreover, $v$ is the identity of $\ey$, so $e=v*e*v\in \ex$. 
\end{proof}

Hence, $\mathfrak{n}_X$ is an isomorphism iff $\mathfrak{n}_Y$ is an isomorphism. Consequently, Newelski's conjecture holds for $X$ iff it holds for $Y$.
\end{proof}
 
\begin{lemma}\label{lemma-Newelski's Conjecture fails for A}
Work in $p$CF. Let the notations be as in Notations \ref{notations-A-V-W-B}.
Then Newelski's conjecture fails for $A$.
\end{lemma}
\begin{proof}
First, by Proposition \ref{prop-standard-A}, we assume that $A$ is definable over $M_0=\Q$. 
By Lemma \ref{lemma-Y00 has finite index}, $A^0 = A^{00}$ has finite index in $\clz{A}(\M)$. Then by Lemma \ref{lemma-new-for-X-for-Y}, Newelski's conjecture fails for $A$ if and only if it fails for $\clz{A}(\M)$. Thus, we assume that $A = \clz{A}(\M)$ for convenience.

 The dfg component $W$ of $A$ then takes the form $W_u \rtimes W_t$, where $\clz{W_u}$ is unipotent, $\clz{W_t}$ is a $M_0$-split torus, and $\clz{W} = \clz{W_u} \rtimes \clz{W_t}$. Since $\clz{A}$ is $M_0$-isotropic, $\Zdim(\clz{W_t}) \geq 1$. Now $W_t$ is isomorphic over $\Q$ to some $n$-fold product $\Gm(\M)^n$ of $\Gm(\M)$, where where $\Gm = (\M \setminus \{0\}, \times)$ denotes the multiplicative group of $\M$ and $n\in{\mathbb N}^{>0}$. By Remark 2.5 of \cite{PPY-sl2qp}, $\Gm(\M)/\Gm(\M)^0$ is infinite.  It follows that $|W/W^0| = |W_t/W_t^0| = |\mathbb{G}_m/\mathbb{G}_m^0|$, which is infinite (in fact, of cardinality continuum).

Let $C = V/W$, and let $f$ be a definable section of the natural projection $V \to C$. By Proposition \ref{pro-III}, the ideal groups of $V = N_G(W)$ are of the form $f(\ec) * \ew$, where $\ec$ is an ideal group of $S_C(M^\ext)$ and $\ew$ is an ideal group of $S_W(M^\ext)$. Since $\ew \cong W/W^0$ is infinite, $f(\ec) * \ew$ is infinite as well. However, $A/A^{00}$ is finite, so Newelski's conjecture fails for $A$.
\end{proof}

Combining Fact \ref{fact-NIP-Amenable-CS} with 
 Lemma \ref{lemma-new-for-X-for-G} and
Lemma \ref{lemma-Newelski's Conjecture fails for A}, we directly obtain the following theorem:

\begin{theorem}\label{thm-Newelski's Conjecture=DA}
Let $G$ be a group definable over a $p$-adically closed field. Then $G$ is definably amenable if and only if Newelski's conjecture holds for $G$.
\end{theorem}

\begin{rmk}
    Note that Theorem \ref{thm-Newelski's Conjecture=DA}  (more precisely, Lemma \ref{lemma-Newelski's Conjecture fails for A}) may fail for groups definable over $\mathbb{R}$: Let $S = \SL(2, \M)$, $P = \PSL(2, \M)$, and let $\pi: S \to P$ denote the natural projection. Since $S = S^{00}$ (see \cite{slr}), it follows that $P^{00} = \pi(S^{00}) = P$. Let $B$ be the Borel subgroup of $S$ consisting of upper triangular matrices; then $B = N_G(B)$ serves as a definably amenable component of $S$. One readily verifies that $\pi(B)$ is a definably amenable component of $P$. Now, $\ker(\pi) = \{I, -I\}$ where $I$ is the identity matrix, and $B = B^{00} \cup -IB^{00}$. Thus, 
    \[
    \pi(B)^{00} = \pi(B^{00}) = B^{00}/\{I, -I\} = (B^{00} \cup -IB^{00})/\{I, -I\} = B/\{I, -I\} = \pi(B).
    \]
    It follows that $\pi(B)^{00} = \pi(B) \cap P^{00}$. By Lemma \ref{lemma-conjecture-equ}, Newelski's conjecture holds for $P$. However, $P$ is not definably amenable.
\end{rmk}


\begin{rmk}
Let notations be as in Proposition \ref{prop-DA-def}. Then $$Z=\{b\in Y\mid G_b \text{ is definably amenable}\}$$ is definable.
Let $$W:=\{b\in Y\mid G_b \text{ satisfies Newelski's conjecture}\}.$$

Theorem \ref{thm-Newelski's Conjecture=DA} gives $W=Z$ in $p$CF, and it follows that $W$ is definable.
We conjecture  that $W$ is also definable in $o$-minimal case.
\end{rmk}

\vspace*{2ex}  

\end{document}